\documentclass[12pt,reqno]{amsart}
\usepackage{amsmath,amsfonts,amssymb,amscd,amsthm,amsbsy,epsf}
\usepackage[dvips]{graphicx}
\usepackage[usenames]{xcolor}
\usepackage{comment}

\footskip=18pt \numberwithin{equation}{section}
\newtheorem{theorem}{Theorem}
\newtheorem{meta-thm}[theorem]{Meta-Theorem}

\newtheorem{proposition}[theorem]{Proposition}

\newtheorem{remark}[theorem]{Remark}
\newtheorem{definition}[theorem]{Definition}

\DeclareMathAlphabet{\mathcalligra}{T1}{calligra}{m}{n}

\newcommand\beq[1]{ \begin{equation}\label{#1} }
\newcommand{\eeq}{ \end{equation} }
\newcommand{\beqno}{ \[ }
\newcommand{\eeqno}{ \] }
\newcommand\beqa[1]{ \begin{eqnarray} \label{#1}}

\newcommand{\eeqa}{ \end{eqnarray} }

\newcommand{\beqano}{ \begin{eqnarray*} }

\newcommand{\eeqano}{ \end{eqnarray*} }

\newcommand\equ[1]{{\rm (\ref{#1})}}

\def\E{{\mathcal E}}

\def\complex{{\mathbb C}}

\begin{document}

\title[The dynamics of the spin-spin problem in Celestial Mechanics]
{The dynamics of the spin-spin problem in Celestial Mechanics}

\author[A.P. Bustamante]{Adrián P. Bustamante}
\address{Department of Mathematics and Physics, University of Roma Tre, Via della Vasca Navale 84, 00146 Roma (Italy)}
\email{adrian.perezbustamante@uniroma3.it}

\author[A. Celletti]{Alessandra Celletti}
\address{
Department of Mathematics, University of Roma Tor Vergata, Via
della Ricerca Scientifica 1, 00133 Roma (Italy)}
\email{celletti@mat.uniroma2.it}

\author[C. Lhotka]{Christoph Lhotka}
\address{
Department of Mathematics, University of Roma Tor Vergata, Via
della Ricerca Scientifica 1, 00133 Roma (Italy)}
\email{lhotka@mat.uniroma2.it}

\baselineskip=18pt              

\thanks{A.B. has been supported by the research project  MIUR-PRIN 2020XBFL "Hamiltonian and Dispersive PDEs". C.L. acknowledges the support from the Excellence Project 2023-2027 MatMod@TOV, the project MIUR-PRIN 20178CJAB "New Frontiers of Celestial Mechanics: Theory and Applications", and the support of GNFM/INdAM}

\maketitle

\begin{abstract}
This work investigates different models of rotational dynamics of two  rigid bodies with the shape of an ellipsoid, moving under their gravitational influence. We assume that the spin axes of the two bodies are perpendicular to the orbital plane and coinciding with the direction of their shortest physical axis. In the basic approximation, we assume that the orbits of the centers of mass are Keplerian and we retain the lowest order of the potential, according to which the rotational motions of the two bodies are decoupled, the so-called \sl spin-orbit \rm problem. When considering highest order approximation of the potential, the rotational motions become coupled giving rise to the so-called \sl spin-spin \rm problem. Finally, we release the assumption that the orbit is Keplerian, the \sl full spin-spin \rm problem, which implies that the rotational dynamics is coupled to the variation of the orbital elements. We also consider the above models under the assumption that one or both bodies are non rigid; the dissipative effect is modeled by a linear function of the rotational velocity, depending on some dissipative and drift coefficients. 

We consider three main resonances, namely the (1:1,1:1), (3:2,3:2), (1:1,3:2) resonances and we start by analyzing the linear stability of the equilibria in the conservative and dissipative settings (after averaging and keeping only the resonant angle), showing that the stability depends on the value of the orbital eccentricity. By a numerical integration of the equations of motion, we compare the spin-orbit and spin-spin problems to highlight the influence of the coupling term of the potential in the conservative and dissipative case. We conclude by investigating the effect of the variation of the orbit on the rotational dynamics, showing that higher order resonant islands, that appeared in the Keplerian case, are destroyed in the full problem. 
\end{abstract}



\section{Introduction}
The dynamics of two bodies orbiting around each other under their mutual gravitational  influence is a topic of interest in Celestial Mechanics, since it might concern a large variety of bodies in the Solar system, most notably planet-satellite pairs. More recently, this interest has grown thanks to some space missions, which have been devoted to the exploration of minor bodies, including binary asteroids, which are typically characterized by irregular shapes. 

Within the two-body orbital-rotational interaction, the simplest model is represented by the \sl spin-orbit coupling, \rm which studies the dynamics of a rigid triaxial ellipsoid around a point mass central planet under simplifying assumptions on the spin-axis, which is taken perpendicular to the orbit plane and coinciding with the shortest physical axis of the ellipsoid (see, e.g., \cite{cel1990}, \cite{cel2010}, \cite{gol1966}). This model has been used in \cite{Colombo} to conjecture the non-synchronous rotation of Mercury or in \cite{wis1984} to conjecture the chaotic rotation of Hyperion. Relaxing the rigidity assumption, one needs to consider a tidal torque as given, e.g., in \cite{Peale}, which brings to the analysis of a dissipative system, that leads to the probability of capture  into resonance (\cite{corlas2004}), the shapes of the basins of attraction (\cite{cel2008}, \cite{CellettiLhotka2014}) and the proof of the existence of quasi-periodic attractors (\cite{cel2009}, \cite{CCGL20a}, \cite{CCGL20b}, \cite{CCGL20c}). 

The irregular shape of minor binary objects requires to consider both interacting bodies with their own shape (see, e.g., \cite{Lu2023}, \cite{Brennan2024}). This problem is known in its generality as the \sl full two-body problem. \rm If we assume that the orbits of the centers of mass of such bodies are Keplerian, then we speak of the \sl Keplerian spin-spin problem. \rm  An exhaustive description of the Keplerian spin-spin problem is given in \cite{mac1995}, see also \cite{sch2002}, \cite{bou2017}, \cite{hou2017}, \cite{housch2017}. In this context, remarkable results have been obtained in \cite{Scheeres}, providing stability conditions, and in \cite{BellroseScheeres2008}, that computes families of periodic orbits and analyzes the relative equilibria. The long-term evolution of the rotational dynamics has been studied in \cite{boulas2009}. An extensive numerical study of the resonances in the full two-body problem is given in \cite{jaf2016}. 

The rotational dynamics of two rigid bodies, as well as the coupling with the orbital motion, might be complex and deserves a thorough investigation. When relaxing the assumption that the orbits are Keplerian, we speak of the \sl full spin-spin problem. \rm Studies of such model are performed, e.g., in \cite{Misra2016}. 

In this work, we intend to analyze the interaction between the rotational dynamics of the two bodies as well as the coupling with their orbits. To this end, we proceed to study models of increasing complexity, from the spin-orbit problem to the Keplerian and full spin-spin models. 
We also consider the dissipative case when one or both bodies are non-rigid; the corresponding dissipation is modeled as in \cite{Peale} with a linear function of the rate of variation of the rotation angle. 

\vskip.1in 

The content of this work is the following. In Section~\ref{sec:spin} we provide the basic notions for the spin-orbit and spin-spin problems, including the definitions of spin-orbit and spin-spin resonances; a special case is given by the synchronous resonance in which the rotational and orbital periods are the same. The equations of motion of the full and Keplerian problems are given in Section~\ref{sec:equations}, while the dissipation is considered in Section~\ref{sub:dissipative}. 
The study of the equilibria and linear stability for the synchronous Keplerian spin-spin resonance, possibly including the dissipation, is given in Section~\ref{sec:res11}, while higher order resonances are shortly analyzed in Section~\ref{sec:res32}. Finally, in Section~\ref{sec:comparison} we provide a comparison between the spin-orbit and spin-spin problems, as well as a numerical investigation of the coupling between the rotational and orbital motions. Some conclusions are given in Section~\ref{sec:conclusions}.

\section{Spin-orbit and spin-spin problems: set-up}\label{sec:spin}

We consider the dynamics of two homogeneous rigid bodies $\E_1$, $\E_2$ having the shape of two ellipsoids and orbiting under their
mutual gravitational attraction. This is a classical problem of Celestial
Mechanics which, in its full generality, is known as the {\sl full
two-body problem} (see, e.g., \cite{Scheeres}). 

\vskip.1in

We consider the dynamics of $\E_1$, $\E_2$ under the following assumptions:

A) The spin axis of each body is perpendicular to the orbital
plane.

B) The direction of the spin axis of each body coincides with the direction of its shortest physical axis.

As a consequence of A) and B), we are assuming that the motion of the centers of mass of the ellipsoids takes place on a plane. 

We refer to the problem described by the assumptions A) and B) as the \sl full spin-spin problem. \rm If, in addition, we consider the following assumption: 

C) The centers of mass of the two ellipsoids move on coplanar
Keplerian orbits with a common focus coinciding with the barycenter of the system, 

then, we refer to the corresponding model as the \sl Keplerian spin-spin problem, \rm since the orbit is Keplerian and it is not coupled with the rotational motion. 

The potential $V_{per}$ describing the gravitational interaction between the two bodies (compare with \equ{V0per} below) will be approximated by the first two terms, say $V_{per}=V_2+V_4$ (using the notation of \cite{CGM2021}), where $V_2$ contains trigonometric terms in which the
rotation angles of the two bodies appear in different terms (hence, the dynamics of each ellipsoid is decoupled from the other body), while
$V_4$ contains terms in which the rotation angles appear in
combination. If we limit the study to the
term $V_2$, we speak of the \sl spin-orbit \rm problem, while if we
consider both terms $V_2$ and $V_4$, we speak of the \sl spin-spin \rm
problem.

\vskip .1in

We introduce the following definition of resonance for a single ellipsoid, like in the spin-orbit problem, in which the two bodies are decoupled.  

\begin{definition}
	We say that the body $\E$ is in a $m:n$ spin-orbit resonance,
	for some non-zero integers $m$, $n$, whenever it makes $m$
	rotations within $n$ orbital revolutions. 
\end{definition}

This definition can be extended to the spin-spin problem as follows. 

\begin{definition}\label{def:res2}
We say that two bodies $\E_1$ and $\E_2$ are in a spin-spin 
resonance of type $(m_1:n_1,m_2:n_2)$ for
non-zero integers $m_1$, $m_2$, $n_1$, $n_2$, whenever the two
ellipsoids are, respectively, in a $m_1:n_1$ and $m_2:n_2$ spin-orbit resonance.
\end{definition}

\section{The equations of motion}\label{sec:equations}

We denote by $M_1$ and $M_2$ the masses of $\E_1$ and $\E_2$, while 
$A_j\leq B_j\leq C_j$, $j=1,2$, are the 
principal moments of inertia associated to the semiaxes that we denote as 
$\mathsf a_j\geq \mathsf b_j\geq \mathsf c_j$.

We introduce normalized units of measure, so that 
\[
  M_1+M_2=1,\quad C_1+ C_2=1,\quad \tau=2\pi\ ,
\]
where $\tau$ is the orbital period of the Keplerian orbit. Then, 
Kepler's third law reads as
\[
  G(M_1+M_2) (\frac{\tau}{2\pi})^2=a^3 \ ,
\]
where $G$ is the gravitational constant, $a$ is the semimajor axis
of the reduced mass of the system, say $\mu=M_1M_2$ in the above units. As a consequence, we have that $G=a^3$ and $n=\sqrt{{{G(M_1+M_2)}\over {a^3}}}=1$.

Next, we introduce the parameters $d_j$ and $q_j$ associated to the ellipsoid $\E_j$ as
$$
  d_j=B_j-  A_j, \qquad
  q_j=2  C_j-   B_j-   A_j \ .
$$
We remark that the quantities $d_j/  C_j$ and $q_j/ C_j$, related to the equatorial oblateness and the flattening of each ellipsoid, are typically small for Solar system objects with regular shape. We also mention the following constraints on the parameters (see \cite{CGM2021}): 
$$
    0\le d_j \le  C_j\le 1,\quad d_j\le q_j\le 2
    C_j\le 2,\quad M_j \mathsf a_j^2=\frac{5}{2}( C_j+d_j)\le
    5 C_j \le 5 \ .
$$

\subsection{The full spin-orbit and spin-spin models}

We start by considering the \sl full \rm spin-orbit
and spin-spin models, without constraining the orbits to be Keplerian ellipses.
We denote by $T$ the kinetic energy that we split as $T=T_{orb}+T_{rot}$, with 
$T_{orb}$ the orbital part and $T_{rot}$ the rotational part; we denote by $V$ the potential energy. We assume that the center of mass of the system coincides with the origin and we identify the orbital plane with the complex plane
$\complex$; we denote by $\vec r_j\in \complex$ the position of $\E_j$. Then, we have 
$$
M_1 \vec r_1+M_2 \vec r_2=\vec 0\ ;
$$
defining $\vec r= \vec r_2-\vec r_1$, we have 
$$
\vec r_1=-M_2 \vec r\ ,\quad \vec r_2=M_1 \vec r\ .
$$
Let $r=|\vec r|$ and let $f$ be the angle between the direction of $\vec r$ and a horizontal reference line, so that $\vec r = r \exp (if)$; we denote by $(\theta_1,\theta_2)$ the rotation angles of $\E_1$ and $\E_2$, namely the angles formed by the longest axis of the ellipsoids, lying on the orbital plane due to assumptions A) and B), and the horizontal reference axis. We can write the kinetic energy as 
\begin{equation*}
  T_{orb}=\frac{1}{2} (M_1 {\dot{\vec r}_1}^2 +M_2 {\dot{\vec r}_2}^2)
  = \frac{\mu}{2} (\dot r^2 + r^2 \dot
  f^2),\quad T_{rot}=\frac{1}{2}  C_1 \dot \theta_1^2 +
  \frac{1}{2}  C_2 \dot \theta_2^2\ .
\end{equation*}
We introduce the momenta conjugated to $r$, $f$, $\theta_1$, $\theta_2$ as
\begin{equation*}
    p_r=\mu \dot r,\quad
    p_f=\mu r^2\dot f,\quad
    p_1 =  C_1 \dot \theta_1,\quad
    p_2 =  C_2 \dot \theta_2\ . 
\end{equation*}
Then, the Hamiltonian can be written as 
$$
    H(p_r,p_f,p_1,p_2,r,f,\theta_1,\theta_2)=\frac{p_r^2}{2\mu}
    + \frac{p_f^2}{2\mu r^2}
    + \frac{p_1^2}{2 C_1} + \frac{p_2^2}{2 C_2}
    +V(\theta_1,\theta_2,r,f)
$$
with associated equations of motion 
\beq{orbital_eqs}
    \dot r = \frac{p_r}{\mu}, \quad
    \dot f = \frac{p_f}{\mu r^2},\quad
    \dot p_r =  \frac{p_f^2}{\mu r^3} - \partial_r V , \quad
    \dot p_f =  - \partial_{f} V
\eeq
for the orbit and
\beq{spin-spin_V}
\dot \theta_j =\frac{p_j}{ C_j},\quad
\dot p_j =  - \partial_{\theta_j} V\ ,\qquad j=1,2 
\eeq
for the rotation. We notice that the coupling between the spin and the orbit is given through the potential $V$, which can be written as 
\begin{equation}\label{V0per}
  V(\theta_1,\theta_2,r,f) = V_0(r) + V_{per}(\theta_1,\theta_2,r,f)
\end{equation}
with $V_0$ the Keplerian potential 
$$
V_0(r)=-\frac{GM_1M_2}{r}\ .
$$
The coupling potential $V_{per}$ can be expanded as $V_{per}=\sum_{l=1}^\infty V_{2l}$ where, following \cite{mis2021} (see also \cite{CGM2021}), the first two terms are given as 
\begin{eqnarray}\label{V024}
    V_2&=&-\frac{GM_2}{4r^3} (q_1+3d_1\cos(2\theta_1-2f))
    -\frac{GM_1}{4r^3} (q_2+3d_2\cos(2\theta_2-2f))\ ,\nonumber\\
    V_4&=&-\frac{3G}{4^3 r^5} \Bigg \{
    12q_1q_2 + \frac{15}{7} [\frac{M_2}{M_1}d_1^2+2\frac{M_2}{M_1}q_1^2 + \frac{M_1}{M_2}d_2^2+2\frac{M_1}{M_2}q_2^2]\nonumber\\
    &&\hspace{1.5cm} +d_1 M_2\left \{ [20\frac{q_2}{M_2} + \frac{100}{7}\frac{q_1}{M_1}]\cos (2\theta_1-2f) +25 \frac{d_1}{M_1}\cos (4\theta_1-4f)
    \right \}\nonumber\\
    &&\hspace{1.5cm}  +d_2 M_1 \left \{[20\frac{q_1}{M_1} + \frac{100}{7}\frac{q_2}{M_2}]\cos (2\theta_2-2f) + 25 \frac{d_2}{M_2}\cos (4\theta_2-4f)
    \right \}\nonumber\\
    &&\hspace{1.5cm} +6d_1 d_2 \cos (2\theta_1-2\theta_2)+70d_1 d_2  \cos (2\theta_1+2\theta_2-4f) \Bigg \}\ .
\end{eqnarray}
With reference to \equ{orbital_eqs} and \equ{spin-spin_V}, where the orbit and the rotation are coupled, we refer to the 

- {\sl full spin-orbit model} if $V_{per}=V_2$, 

- {\sl full spin-spin model} if $V_{per}=V_2+V_4$.

We remark that, as noticed in \cite{CGM2021}, the total angular momentum $P_f=p_f+p_1+p_2$ is a constant of motion.

\subsection{The Keplerian spin-orbit and spin-spin models}\label{sec:Kepler}
In the {\sl Keplerian spin-orbit model} and {\sl  Keplerian spin-spin model}, we constrain the orbit to be Keplerian, which means that in \equ{orbital_eqs} we limit the potential to $V_0$ (which depends just on the radius) and in 
\equ{spin-spin_V} we consider only $V_{per}$ (since $V_0$ does not depend on $\theta_j$), where $r$ and $f$ are the solutions of \equ{orbital_eqs}:
\beqa{orbital_eqs_kep}
  \dot r &=& \frac{p_r}{\mu}, \quad
  \dot f = \frac{p_f}{\mu r^2},\quad \dot p_r =  \frac{p_f^2}{\mu r^3} - \partial_r V_0 , \quad
  \dot p_f = 0,\nonumber\\
  \dot \theta_j &=&\frac{p_j}{ C_j},\quad
  \dot p_j = - \partial_{\theta_j}V_{per}\ ,\qquad j=1,2\ .
\eeqa
The equations for the orbital motion are the classical solutions of 
Kepler's problem, which describe an ellipse with semimajor axis $a$ and 
eccentricity $e$. Due to the normalization of the units of measure, the mean anomaly coincides with the time and we can express the orbital solution as 
$$
    r=r(t;a,e),\quad f=f(t;e),\quad
    p_r=p_r(t;a,e),
$$
supplemented by the first integral 
$$
p_f=p_f(a,e)=\mu a^2\sqrt{1-e^2}\ . 
$$
The orbital motion is described by the following 2D, time dependent Hamiltonian: 
\begin{equation}\label{HK}
H_K(p_1,p_2,\theta_1,\theta_2,t)={{p_1^2}\over {2
    C_1}}+{{p_2^2}\over {2 C_2}}+V_{per}(\theta_1,\theta_2,t)\ .
\end{equation}
We observe that, if we take $V_{per}=V_2$, then we obtain two uncoupled Keplerian spin-orbit models; if we take $V_{per}=V_2+V_4$, we have a coupling between the rotational motions of the two ellipsoids. 

For later use, it is useful to define the following non-dimensional parameters:
$$
    \lambda_j = 3\frac{\mu}{M_j} \frac{d_j}{ C_j} \ , \qquad
    \sigma_j=\frac{1}{3}\frac{ C_j}{\mu a^2},\qquad \hat
    q_j=\frac{q_j}{M_j a^2}\ ,
$$
which are limited by the constraint $ C_1\sigma_2= C_2\sigma_1$.

\section{The dissipative Keplerian spin-orbit and spin-spin models}\label{sub:dissipative}
In this Section, we consider the Keplerian spin-orbit and spin-spin models by adding the dissipative effect due to the tidal torque. In particular, we consider \equ{HK} with $V_{per}=V_2+V_4$:
\beqa{diss}
C_1\ddot\theta_1+({{\partial V_2}\over {\partial\theta_1}}+{{\partial V_4}\over {\partial\theta_1}})&=&
-\delta_1 C_1({a\over {r(t)}})^6\ (\dot\theta_1-\dot f(t))\nonumber\\
C_2\ddot\theta_2+({{\partial V_2}\over {\partial\theta_2}}+{{\partial V_4}\over {\partial\theta_2}})&=&
-\delta_2 C_2({a\over {r(t)}})^6\ (\dot\theta_2-\dot f(t))\ ,
\eeqa
where we added the dissipation in the form of MacDonald torque (see, e.g., \cite{Peale}) at the right hand sides of \equ{diss};
the dissipation depends on the dissipative constants $\delta_1$, $\delta_2$, which are typically small
with respect to the gravitational part.

\vskip.1in

We can possibly replace the right hand sides of \equ{diss} by their averages over one orbital period, thus obtaining
the following equations:
\beqa{dissave}
{ C_1}\ddot\theta_1+({{\partial V_2}\over {\partial\theta_1}}+{{\partial V_4}\over {\partial\theta_1}})&=&
-\bar\gamma_1\ (\dot\theta_1-\bar\mu_1)\nonumber\\
{ C_2}\ddot\theta_2+({{\partial V_2}\over {\partial\theta_2}}+{{\partial V_4}\over {\partial\theta_2}})&=&
-\bar\gamma_2\ (\dot\theta_1-\bar\mu_2)
\eeqa
with $\bar\gamma_1$, $\bar\gamma_2$, $\bar\mu_1$, $\bar\mu_2$ constants, whose explicit expressions are
\beqa{gm}
\bar\gamma_j&=&\delta_j{ C_j}{1\over {(1-e^2)^{9\over 2}}}\ (1+3e^2+{3\over 8}e^4)\ ,\nonumber\\
\bar\mu_j&=&{n\over {(1-e^2)^{3\over 2}}}\ {{1+{{15}\over 2}e^2+{{45}\over 8}e^4+{5\over {16}}e^6}\over {1+3e^2+{3\over 8}e^4}}\ ,\qquad j=1,2\ ,
\eeqa
where we remind that $n$ denotes the mean motion. The derivation of \equ{gm} is presented in Appendix~\ref{app:gm}.
We refer to \equ{diss} or \equ{dissave} as \sl double-dissipative \rm systems, since the dissipation acts on both bodies. 
We will also consider \sl single-dissipative \rm systems by assuming that the dissipation acts directly only on one body, 
which corresponds to take, for example, $\bar\gamma_1$, $\bar\mu_1\not=0$ and $\bar\gamma_2=\bar\mu_2=0$.

\section{Equilibria and linear stability for the (1:1,1:1) Keplerian spin-spin resonance}\label{sec:res11}
Let us write the equations of motion of the averaged double-dissipative Keplerian spin-orbit and spin-spin problems in the form
\beqano
{ C_1} \ddot\theta_1  + \frac{\partial V^{(\chi)}}  {\partial \theta_1} &=& -\bar\gamma_1(\dot\theta_1 -\bar\mu_1)\nonumber\\
{ C_2} \ddot\theta_2 + \frac{\partial V^{(\chi)}}{\partial \theta_2} &=& -\bar\gamma_2(\dot \theta_2 -\bar\mu_2)\ ,
\eeqano
where we introduced the function $V^{(\chi)}(\theta_1, \theta_2,t)  = V_2(\theta_1, \theta_2,t) +\chi V_4(\theta_1, \theta_2,t)$; we recover the spin-orbit problem by taking $\chi=0$
and the spin-spin model by taking $\chi=1$.
The expansions of $V_2$ and $V_4$ will be considered up to second order in the eccentricity as given in Appendix \ref{sub:potential}.

To deal with the  (1:1,1:1) spin-spin resonances as in Definition \ref{def:res2}, we introduce the
angles $$ \varphi_1 = 2(\theta_1-t), \qquad \varphi_2 = 2(\theta_2
-t)\ .$$
The equations of
motion considering just the terms that depend on
$\varphi_1$ and $\varphi_2$, namely the terms that do not vanish
after averaging with respect to time, become
\beqa{eqall}
\frac{C_1}{2}\ddot\varphi_1 + K_1 \sin(\varphi_1) + \chi \{ K_2\sin(\varphi_1) + K_3\sin(2\varphi_1) +
K_4\sin(\varphi_1 + \varphi_2) &+& K_5\sin(\varphi_1 - \varphi_2) \} \nonumber\\
&=& -\bar\gamma_1\left(\frac{\dot\varphi_1}{2} + 1 -\bar\mu_1\right) \nonumber\\
\frac{C_2}{2}\ddot\varphi_2 + L_1 \sin(\varphi_2) + \chi \{ L_2\sin(\varphi_2) + L_3\sin(2\varphi_2) +
L_4\sin(\varphi_1 + \varphi_2) &-&L_5\sin(\varphi_1 -\varphi_2) \} \nonumber\\
&=& -\bar\gamma_2\left(\frac{\dot\varphi_2}{2} + 1 -\bar\mu_2\right).\nonumber\\
\eeqa
where, from Appendix \ref{sub:potential}, we have 
\beqa{KL}
K_1 &=& \frac{3d_1GM_2}{2a^3}(1-\frac{5}{2}e^2),\qquad\qquad\ \ L_1 =
\frac{3d_2GM_1}{2a^3}(1-\frac{5}{2}e^2)\nonumber\\
K_2&=&\frac{75d_1GM_2q_1}{56a^5M_1} + \frac{75d_1e^2GM_2q_1}{56a^5M_1} +\frac{15d_1Gq_2}{8a^5} + \frac{15d_1e^2Gq_2}{8a^5} \nonumber \\
L_2 &=& \frac{75d_2GM_1q_2}{56a^5M_2} + \frac{75d_2e^2GM_1q_2}{56a^5M_2} + \frac{15d_2Gq_1}{8a^5} + \frac{15d_2e^2Gq_1}{8a^5}\nonumber\\
K_3&=& \frac{75d_1^2G M_2}{16a^5M_1}(1-11e^2),\qquad\qquad\ L_3 = \frac{75d_2^2GM_1}{16a^5M_2}(1-11e^2) \nonumber\\
K_4 &=& \frac{105d_1d_2G}{16a^5}(1-11e^2), \qquad\qquad\ L_4
= \frac{105d_1d_2G}{16a^5}(1-11e^2)\nonumber\\
K_5 &=& \frac{9d_1d_2G}{16a^5} + \frac{45d_1d_2e^2G}{16a^5},\qquad\qquad L_5 = \frac{9d_1d_2G}{16a^5} + \frac{45d_1d_2e^2G}{16a^5} .
\eeqa
We notice that $K_2$, $L_2$, $K_5$, $L_5$ are always positive; $K_1$, $L_1$ are positive for $e<\sqrt{2\over 5}$; $K_3$, $L_3$, $K_4$, $L_4$ are positive for $e<{1\over {\sqrt{11}}}$. In the following, 
we analyze separately the linear stability of the conservative (Section \ref{cons11}) and dissipative (Section \ref{diss11}) case.

\subsection{Conservative case, (1:1,1:1) resonance}\label{cons11}
In the conservative case, we set $\bar\gamma_1=\bar\gamma_2 = 0$ and we write the equations of motion \equ{eqall} as 
\beqa{eq1}
    \dot\varphi_1 &=& \frac{2}{C_1} J_1 \nonumber\\
    \dot\varphi_2 &=& \frac{2}{C_2} J_2 \nonumber\\
    \dot J_1 &=& - K_1 \sin(\varphi_1) - \chi \{ K_2\sin(\varphi_1) + K_3\sin(2\varphi_1) + K_4\sin(\varphi_1 + \varphi_2) +
    K_5\sin(\varphi_1 -\varphi_2) \} \nonumber\\
    \dot J_2 &=&-L_1 \sin(\varphi_2) - \chi \{ L_2\sin(\varphi_2) +
    L_3\sin(2\varphi_2) + L_4\sin(\varphi_1 + \varphi_2) - L_5\sin(\varphi_1 -\varphi_2) \} \label{eqcons}\ .\nonumber\\
\eeqa

The equilibrium points are located at $J_1=J_2 = 0$, while the angles are the solutions of the equations
\beqa{eq:crit-conservative}
    K_1 \sin(\varphi_1) + \chi \{ K_2\sin(\varphi_1) + K_3\sin(2\varphi_1) + K_4\sin(\varphi_1 + \varphi_2) +
    K_5\sin(\varphi_1 - \varphi_2) \}  &=& 0 \nonumber\\
    L_1 \sin(\varphi_2) + \chi \{ L_2\sin(\varphi_2) + L_3\sin(2\varphi_2) + L_4\sin(\varphi_1 + \varphi_2) - 
    L_5\sin(\varphi_1 - \varphi_2) \} &=& 0\ .\nonumber\\
\eeqa

\begin{remark}
Notice that $(\varphi_1,\varphi_2) = (0,0)$, $(0,\pi)$, $(\pi,0)$, $(\pi,\pi)$ are solutions of \equ{eq:crit-conservative}. The existence of other solutions of \eqref{eq:crit-conservative} might be possible and depend on the coefficients appearing in \eqref{eq:crit-conservative}.
\end{remark}

We analyze the character of the equilibria by linearizing the equations of motion \equ{eqcons}.
To this end, let $(\varphi_1^0,\varphi_2^0)$ be a solution of \equ{eq:crit-conservative} and set
$$
\varphi_1=\varphi_1^0+\xi_1\ ,\qquad \varphi_2=\varphi_2^0+\xi_2
$$
for $\xi_1$, $\xi_2$ small. From \equ{eqcons}, we have
the following linearized equations of motion at the equilibrium point $J_1=0$, $J_2=0$, $\varphi_1=\varphi_1^0$, and
$\varphi_2=\varphi_2^0$:
\beqa{eqlincons}
    \dot\xi_1 &=& \frac{2}{C_1}J_1 \nonumber\\
    \dot\xi_2 &=& \frac{2}{C_2}J_2 \nonumber\\
    \dot J_1 &=& -\left\{ (K_1 + \chi K_2)\cos(\varphi_1^0) +2\chi K_3\cos(2\varphi_1^0) +
    \chi K_4 \cos(\varphi_1^0 + \varphi_2^0) + \chi K_5\cos(\varphi_1^0-\varphi_2^0) \right\}\xi_1 \nonumber\\
    &\quad& - \chi\left\{ K_4 \cos(\varphi_1^0 + \varphi_2^0) - K_5\cos(\varphi_1^0-\varphi_2^0) \right\} \xi_2 \nonumber\\
    \dot J_2 &=& - \chi\left\{ L_4 \cos(\varphi_1^0 + \varphi_2^0) - L_5\cos(\varphi_1^0-\varphi_2^0) \right\} \xi_1 \nonumber\\
    & \quad& -\left\{ (L_1 + \chi L_2)\cos(\varphi_2^0) +2\chi L_3\cos(2\varphi_2^0) + \chi L_4 \cos(\varphi_1^0
    + \varphi_2^0) +\chi L_5 \cos(\varphi_1^0-\varphi_2^0) \right\}\xi_2\ .\nonumber\\
\eeqa

The characteristic polynomial and the eigenvalues of \equ{eqlincons} depend on the
fixed points and the constants $K_i$, $L_i$.
The equilibria are either $(\varphi_1, \varphi_2, J_1, J_2) = (0, 0, 0, 0)$, $(\pi, 0, 0, 0)$,
$(0, \pi, 0, 0)$, $(\pi, \pi, 0, 0)$ or rather the point $(\bar\varphi_1,\bar\varphi_2, 0, 0)$
with $(\bar\varphi_1,\bar\varphi_2)$ solution of \equ{eq:crit-conservative}.

\subsubsection{Linear stability of the origin}
To provide an explicit example, let us consider the critical point at $(\varphi_1, \varphi_2,
J_1, J_2) = (0, 0, 0, 0)$; in this case the linear system is given by
\beqa{EMlin}
    \dot\xi_1 &=& \frac{2}{C_1}J_1 \nonumber\\
    \dot\xi_2 &=& \frac{2}{C_2}J_2 \nonumber\\
    \dot J_1 &=& -\left\{ (K_1 + \chi K_2) +2\chi K_3 + \chi K_4 + \chi K_5 \right\}\xi_1 - \chi
    \left\{ K_4  - K_5 \right\} \xi_2\nonumber\\
    \dot J_2 &=& - \chi\left\{ L_4 - L_5 \right\} \xi_1  -\left\{ (L_1 + \chi L_2) +2\chi L_3 + \chi L_4 +
    \chi L_5  \right\}\xi_2\ .
\eeqa
We start by analyzing the eigenvalues of the spin-orbit dynamics with $\chi=0$.

\begin{proposition}\label{pro:SO}
If $e<\sqrt{2\over 5}$, then the eigenvalues of the linearized uncoupled spin-orbit motion with $\chi=0$ in \equ{EMlin} for the critical point $(\varphi_1,\varphi_2,J_1,J_2)=(0,0,0,0)$ are purely imaginary. 
\end{proposition}

\begin{proof}
Taking $\chi=0$ in \equ{EMlin}, we obtain two uncoupled systems described by the equations
\beqa{SOk}
\dot\xi_k&=&{2\over C_k} J_k\nonumber\\
\dot J_k&=&-Q_k\xi_k\ ,\qquad k=1,2\ ,
\eeqa
where $Q_1=K_1$, $Q_2=L_1$. The eigenvalues associated to \equ{SOk} are the solutions of the characteristic equation
$$
(x^2+{{2Q_1}\over C_1})(x^2+{{2Q_2}\over C_2})=0\ ,
$$
which are purely imaginary if $e<\sqrt{2\over 5}$, which implies $Q_j>0$, $j=1,2$. 
\end{proof}

Next, we analyze the eigenvalues of the spin-spin problem with $\chi=1$. Let us introduce the following notation:
\beqa{ab12}
a_1 &:=& -(K_1+ K_2 + 2K_3 + K_4 + K_5)\ ,\nonumber\\
a_2 &:=& -(L_4 - L_5)\ ,\nonumber\\
b_1 &:=&-(K_4-K_5)\ ,\nonumber\\
b_2 &:=& -(L_1 + L_2 + 2L_3 + L_4 + L_5)\ .
\eeqa

Let $\alpha$, $\beta$ be defined as
\beq{ab2}
\alpha= -2(\frac{a_1}{C_1} + \frac{b_2}{C_2})\ , \qquad \beta=\frac{4}{C_1C_2}(a_1b_2 -a_2b_1)\ .
\eeq

\begin{proposition}\label{pro:SS}
If $e<{1\over {\sqrt{11}}}$, then the eigenvalues of the linearized spin-spin motion for the $(1:1,1:1)$ resonance with $\chi=1$ in \equ{EMlin} are purely imaginary. 
\end{proposition}

\begin{proof}
The matrix associated to the linearized equations \equ{EMlin} is
$$\begin{pmatrix}
    0 &0& \frac{2}{C_1} & 0 \\
    0 & 0 & 0& \frac{2}{C_2} \\
    a_1 & b_1 & 0 &0 \\
    a_2 & b_2 & 0 &0
\end{pmatrix}
$$
and the characteristic polynomial is given by 
\beq{seceq}
x^4 + \alpha x^2 + \beta = 0 
\eeq
with $\alpha$, $\beta$ as in \equ{ab2}. Therefore, the eigenvalues are given by
\beq{eigenvalues}
x^2 = {{-\alpha \pm \sqrt{\alpha^2 - 4\beta}}\over 2}\ .
\eeq
Straightforward computations give that $a_1<0$ and $b_2<0$ for $e<{1\over {\sqrt{11}}}$, which implies $\alpha>0$. 
We also notice that we can write 
\beq{511bis}
\alpha^2-4\beta = 4\left(\frac{a_1}{C_1} - \frac{b_2}{C_2}\right)^2 + 16\frac{a_2b_1}{C_1C_2} >0\ ,
\eeq
since $({a_1\over C_1}-{b_2\over C_2})^2>0$ and $a_2b_1>0$, being $a_2b_1=(L_4-L_5)(K_4-K_5)=(K_4-K_5)^2$; this implies that the right hand side of \eqref{eigenvalues} is always negative and, therefore, the eigenvalues are purely imaginary.
\end{proof}

From Propositions~\ref{pro:SO} and \ref{pro:SS} we notice an important difference between the spin-orbit and the spin-spin problems, since it occurs that the eigenvalues are purely imaginary, and hence the origin is linearly stable, for $e\lesssim 0.632$ for the spin-orbit problem, while we get a smaller eccentricity, say $e\lesssim 0.301$, for the spin-spin problem. 

\vskip.1in 

A similar analysis can be implemented for the other equilibrium points, say $(\varphi_1^0,\varphi_2^0)$, although for equilibria different from $(0,0)$, the sign of the coefficients 
\beqa{ab}
	a_1 &:=& -((K_1+K_2)\cos(\varphi_1^0) + 2K_3\cos(2\varphi_1^0) + K_4\cos(\varphi_1^0+\varphi_2^0) + K_5\cos(\varphi_1^0 - \varphi_2^0),\nonumber\\
	a_2 &:=& -(L_4\cos(\varphi_1^0 +\varphi_2^0) - L_5\cos(\varphi_1^0 - \varphi_2^0)),\nonumber\\
	b_1 &:=& -(K_4\cos(\varphi_1^0 + \varphi_2^0) - K_5\cos(\varphi_1^0 - \varphi_2^0)),\nonumber\\
	b_2 &:=& -((L_1 + L_2)\cos(\varphi_2^0) + 2L_3\cos(2\varphi_2^0) + L_4\cos(\varphi_1^0 + \varphi_2^0) + L_5\cos(   \varphi_1^0 - \varphi_2^0))\ ,\nonumber\\
\eeqa
and, hence, of $\alpha$, $\beta$ in \equ{511bis} depends on the quantities $d_j$, $q_j$, $M_j$ appearing in \equ{KL}. With this motivation, we give a concrete example in Figure~\ref{fig:res11}, which refers to the binary asteroid Patroclus and Menoetius, for which (in our units) $d_1=0.0482$, $d_2=0.0321$, $q_1=0.2226$, $q_2=0.1443$, $M_1=0.56$, $M_2=0.44$, $C_1=0.6$, $C_2=0.4$ (see, e.g., \cite{mis2021}). In Figure~\ref{fig:res11}, we plot the maximum of the real part of the eigenvalues of the Keplerian spin-spin problem (eq.s \equ{eq1} with $\chi=1$), 
using a color scale for the equilibria $(0,\pi)$ (left panel); the results for the equilibria $(\pi,0)$ and $(\pi,\pi)$ are very similar. We give the results in a range of values for the semimajor axis between 15 and 30, and the eccentricity between 0 and 0.3. The astronomical values of these quantities for Patroclus and Menoetius are  $a=18.2$, $e=0.02$. The figure shows that the linear stability is, in general, independent from the eccentricity, while the size of the maximum of the eigenvalues changes with the semimajor axis.

\begin{figure}[h]
	\center
	\includegraphics[trim = 8cm  3cm 3.5cm  3cm, scale=0.13]{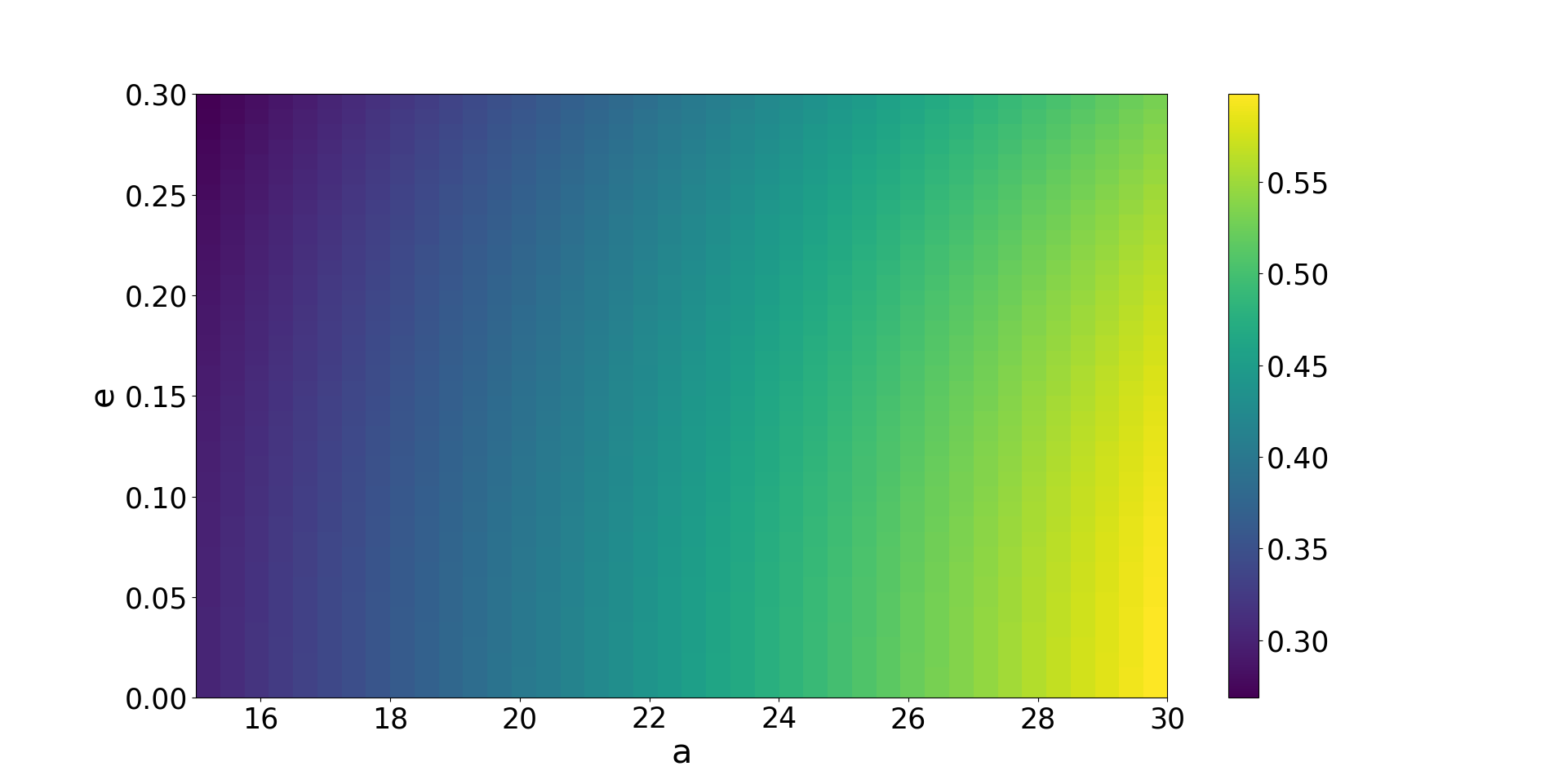}
	\includegraphics[trim = 5cm  3cm 2.5cm  3cm, scale=0.13]{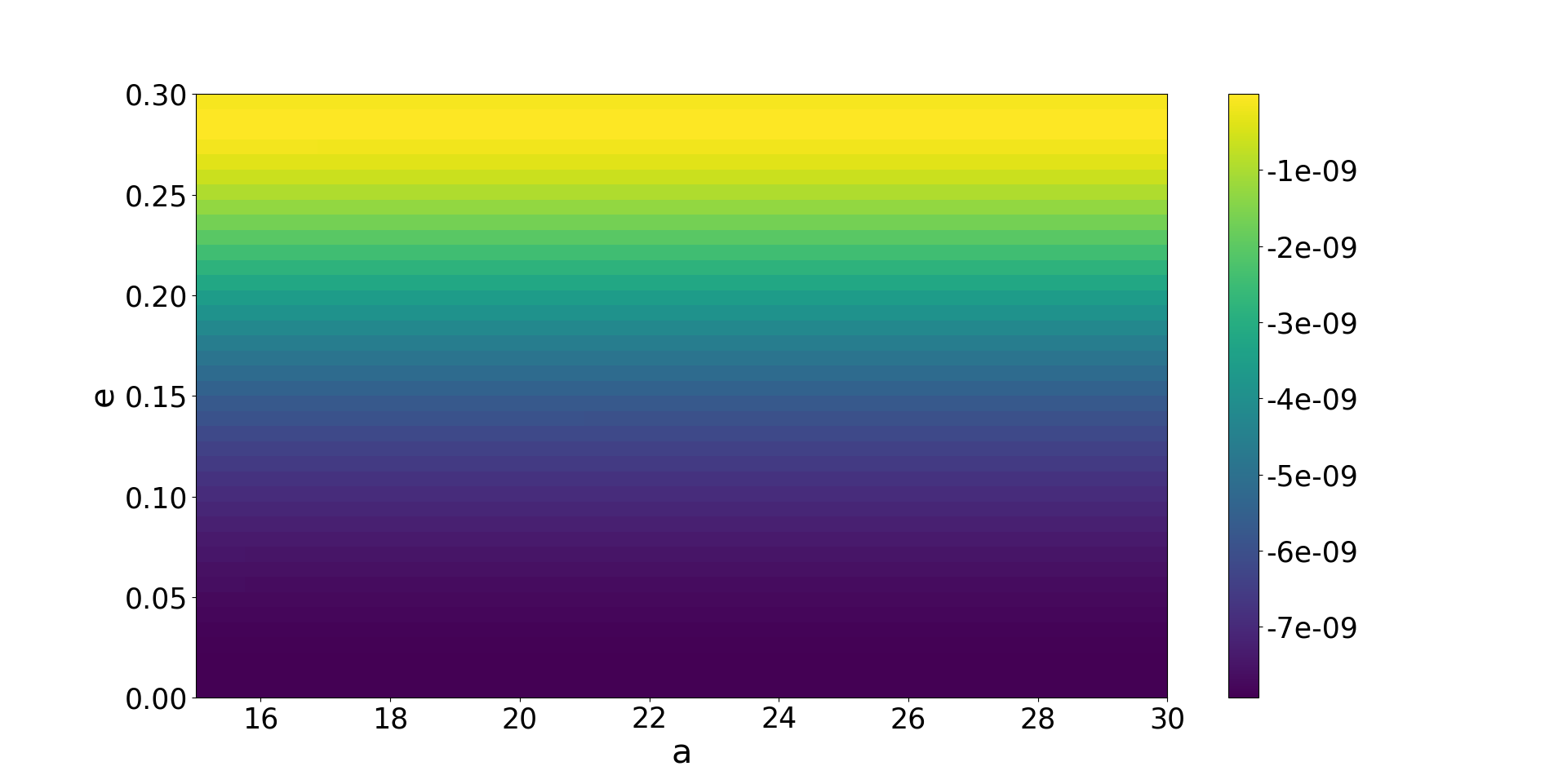}
	\includegraphics[trim = 5cm  3cm 7cm  3cm, scale=0.13]{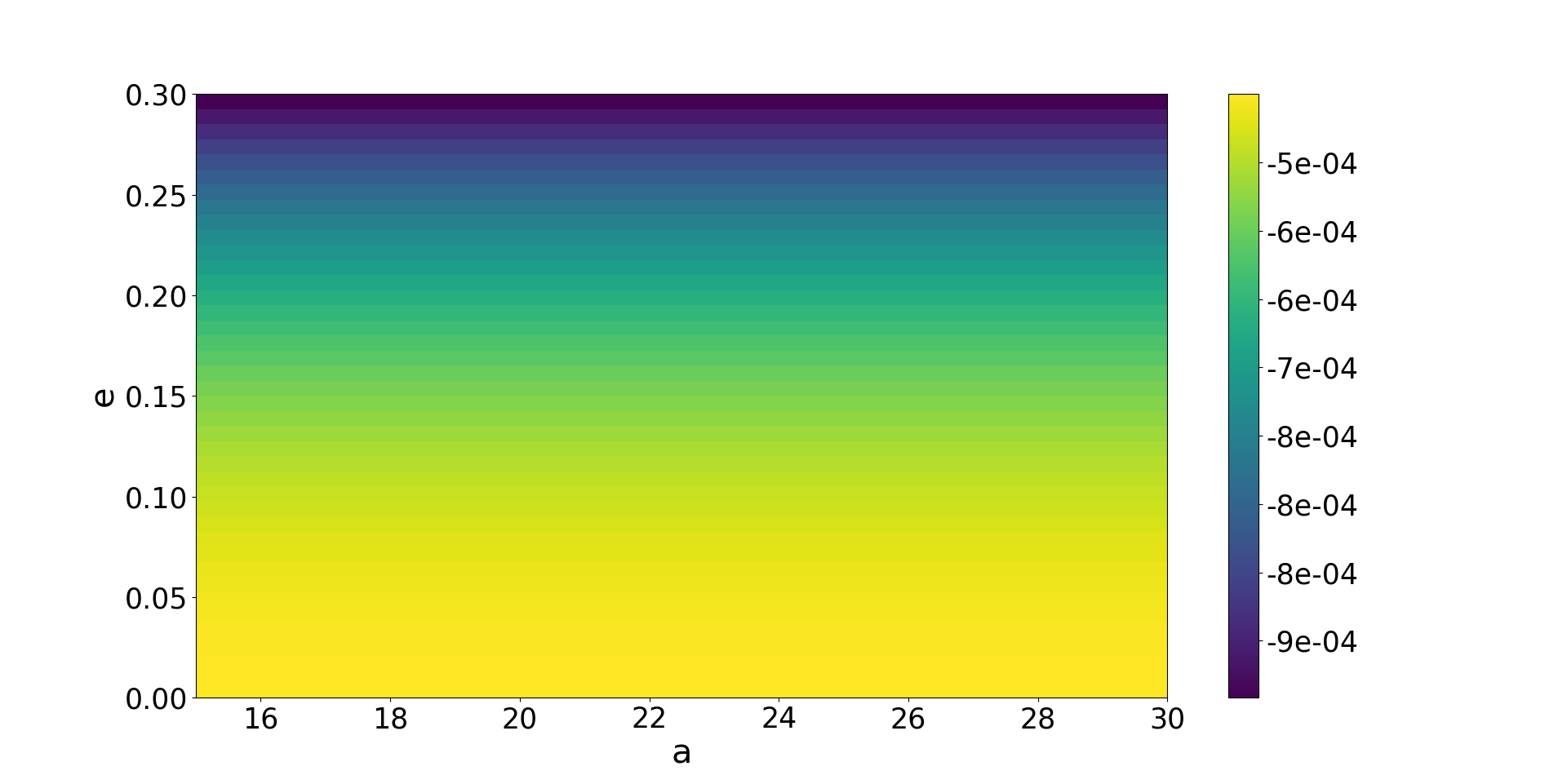}
	\caption{Keplerian spin-spin problem, maximum of the real part of the eigenvalues on a color scale for the conservative case, $(0,\pi)$ (left panel), single-dissipative averaged case, $(0,0)$ with $\delta_1=10^{-3}$ (middle panel), double-dissipative averaged case, $(0,0)$ with $\delta_1=10^{-3}$ and $\delta_2=2\cdot 10^{-3}$ (right panel). 
	}\label{fig:res11}
\end{figure}

\subsection{Dissipative case, (1:1,1:1) resonance}\label{diss11}
Let us consider the dissipative case with $\bar\gamma_1\neq 0$, $\bar\gamma_2\neq 0$ in \equ{eqall}.
The equilibrium points are at $J_1=J_2=0$ and at the solutions of
the system of equations 
\begin{align}
    K_1 \sin(\varphi_1) + \chi \{ K_2\sin(\varphi_1) + K_3\sin(2\varphi_1) + K_4\sin(\varphi_1 + \varphi_2) &+
    K_5\sin(\varphi_1 -\varphi_2)\}\nonumber\\ 
    &+ \bar\gamma_1\left( 1 -\bar\mu_1\right) = 0 \nonumber\\
    L_1 \sin(\varphi_2) + \chi \{ L_2\sin(\varphi_2) + L_3\sin(2\varphi_2) + L_4\sin(\varphi_1 + \varphi_2) &-
    L_5\sin(\varphi_1 -\varphi_2) \}\nonumber\\ 
    &+ \bar\gamma_2\left( 1 -\bar\mu_2\right) = 0\ . \label{eq:crit-dis-1:1}
\end{align}

\begin{proposition}\label{pro:SS}
Let $(\varphi_1^0,\varphi_2^0)$ be an equilibrium point, solution of the system of equations \equ{eq:crit-dis-1:1} for the spin-orbit problem with $\chi=0$. If 
$$
K_1\ \cos\varphi_1^0>0\ ,\qquad L_1\ \cos\varphi_2^0>0
$$
and if the dissipative constants $\bar\gamma_1$, $\bar\gamma_2$ satisfy 
\beq{g12}
0\leq\bar\gamma_1<(8K_1 C_1 \cos\varphi_1^o)^{1\over 2}\ ,\qquad 
0\leq\bar\gamma_2<(8L_1 C_2 \cos\varphi_2^o)^{1\over 2}\ ,
\eeq
then, the eigenvalues are complex conjugate with negative real part. 
\end{proposition}

\begin{proof}
The linearized equations of motion at the equilibrium points
$J_1=0$, $J_2=0$, $\varphi_1=\varphi_1^0$, and
$\varphi_2=\varphi_2^0$ are 
\beqa{eqdiss}
    \dot\xi_1 &=& \frac{2}{C_1}J_1 \label{eq:lin-sys-dis}\nonumber \\
    \dot\xi_2 &=& \frac{2}{C_2}J_2\nonumber \\
    \dot J_1 &=& -\left\{ (K_1 + \chi K_2)\cos(\varphi_1^0) +2\chi K_3\cos(2\varphi_1^0) + \chi K_4 \cos(\varphi_1^0 +
    \varphi_2^0) + \chi K_5\cos(\varphi_1^0-\varphi_2^0) \right\}\xi_1\nonumber \\
    &&\quad - \chi\left\{ K_4 \cos(\varphi_1^0 + \varphi_2^0) - K_5\cos(\varphi_1^0-\varphi_2^0) \right\} \xi_2  -
    \frac{\bar\gamma_1}{C_1} J_1\nonumber \\
    \dot J_2 &=& - \chi\left\{ L_4 \cos(\varphi_1^0 + \varphi_2^0) - L_5\cos(\varphi_1^0-\varphi_2^0) \right\} \xi_1 \\
    && \quad -\left\{ (L_1 + \chi L_2)\cos(\varphi_2^0) +2\chi L_3\cos(2\varphi_2^0) +
    \chi L_4 \cos(\varphi_1^0 + \varphi_2^0) +\chi L_5 \cos(\varphi_1^0-\varphi_2^0) \right\}\xi_2\nonumber \\ 
    &&\quad - \frac{\bar\gamma_2}{C_2}J_2. \nonumber
\eeqa
The characteristic polynomial and the eigenvalues associated to \equ{eqdiss} depend on the values $(\varphi_1^0,\varphi_2^0)$ and the constants $K_i$, $L_i$, $i=1,...,5$.
For $\chi=0$ (spin-orbit problem), the characteristic polynomial is given by
$$
\left(x^2+{\bar \gamma_1\over C_1} x  + {{2K_1}\over C_1} \cos(\varphi_1^0)\right)\left(x^2+{\bar \gamma_2\over C_2} x + {{2L_1}\over C_2} \cos(\varphi_2^0)\right)=0,
$$
whose solutions are
$$
x_{1,2}=-{\bar\gamma_1\over {2C_1}}\pm{1\over 2}\sqrt{\left({\bar\gamma_1\over C_1}\right)^2-8{K_1\over C_1}\cos(\varphi_1^0)}\ ,
$$
$$x_{3,4}=-{\bar\gamma_2\over {2C_2}}\pm{1\over 2}\sqrt{\left({\bar\gamma_2\over C_2}\right)^2-8{L_1\over C_2}\cos(\varphi_2^0)}\ .$$
If $\bar\gamma_1$, $\bar\gamma_2$ satisfy \equ{g12}, then the eigenvalues are complex conjugate with negative real part.
\end{proof}

Proposition \ref{pro:SS} shows that, under dissipation, the equilibrium is stable, provided the dissipative coefficients are small enough; indeed, the inequalities \equ{g12} are satisfied for most objects of the Solar system. 

A concrete example of the computation of the eigenvalues for the single and double dissipative cases for the pair Patroclus-Menoetius is given in Figure \ref{fig:res11} (middle and right panels), showing that the size of the maximum eigenvalue increases from the single to the double dissipative case. 
We remark that in these plots $\bar\gamma_1$, $\bar\gamma_2$ are chosen to satisfy the inequalities \equ{g12}. Beside, we notice that, contrary to the conservative case, in the two dissipative cases shown in Figure~\ref{fig:res11}, the linear stability varies with the eccentricity and it is essentially constant with respect to the semimajor axis. 
 
 \begin{remark}
Taking $\chi =1$ (spin-spin problem), we have that,
defining the coefficients $a_1$, $a_2$, $b_1$, $b_2$ as in \equ{ab}, 
the characteristic equation is given by
$$
\left( x^2 + \frac{\bar\gamma_1}{C_1}x - \frac{2a_1}{C_1}
\right)\left( x^2 + \frac{\bar\gamma_2}{C_2}x - \frac{2b_2}{C_2}
\right) - \frac{4}{C_1 C_2}a_2 b_1=0\ ,
$$
whose solutions are not elementary and depend on the values of the parameters. 
\end{remark}

\section{Equilibria and linear stability for the (3:2,3:2) and (1:1,3:2) Keplerian spin-spin resonances}\label{sec:res32}

To deal with the  (3:2,3:2) spin-spin resonance, we introduce the angles 
$$ 
\varphi_1 = 2\theta_1-3t, \qquad \varphi_2 = 2\theta_2 - 3t\ .
$$ 
The equations of motion considering just the terms that depend only on $\varphi_1$ and $\varphi_2$ and their combinations, namely the terms that will not vanish after averaging with respect to $t$, become
\beqa{diss32}
	\frac{C_1}{2}\ddot\varphi_1  + K_1\sin(\varphi_1) + \chi \left\{K_2 \sin(\varphi_1) + K_3\sin(2\varphi_1) + K_4\sin(\varphi_1 + \varphi_2) + K_5\sin(\varphi_1 - \varphi_2) \right\} \nonumber \\ = -\bar\gamma_1\left(\frac{\dot\varphi_1}{2} + \frac{3}{2} -\bar\mu_1\right) \nonumber \\
	\frac{C_2}{2}\ddot\varphi_2 + L_1\sin(\varphi_2) + \chi \left\{L_2 \sin(\varphi_1) + L_3\sin(2\varphi_2) + L_4\sin(\varphi_1 + \varphi_2) - L_5\sin(\varphi_1 - \varphi_2) \right\} \nonumber \\= -\bar\gamma_2\left(\frac{\dot\varphi_2}{2} + \frac{3}{2} -\bar\mu_2\right)\ ,\nonumber \\
\eeqa
where, from Appendix \ref{sub:potential}, we have 
\beqa{kappaelle}
K_1 &=& \frac{21d_1eGM_2}{4a^3},\qquad\qquad\qquad\qquad\quad\ L_1 = \frac{21d_2eGM_1}{4a^3}\nonumber\\ 
K_2 &=&\frac{675d_1eGM_2q_1}{112a^5M_1} +\frac{135d_1eGq_2}{16a^5}, \qquad L_2 = \frac{675d_2eGM_1q_2}{112a^5M_2} + \frac{135d_2eGq_1}{16a^5}\nonumber\\ 
K_3 &=& \frac{3825d_1^2e^2GM_2}{32a^5M_1}, \qquad\qquad\qquad\qquad L_3 = \frac{3825d_2^2e^2GM_1}{32a^5M_2}\nonumber\\  
K_4&=&\frac{5355d_1d_2e^2G}{32a^5},\qquad\qquad\qquad\qquad\ \ L_4 = \frac{5355d_1d_2e^2G}{32a^5}\nonumber\\ 
K_5&=& \frac{45d_1d_2e^2G}{16a^5} + \frac{9d_1d_2G}{16a^5}, \qquad\qquad\quad L_5 = \frac{45d_1d_2e^2G}{16a^5} + \frac{9d_1d_2G}{16a^5}\ .
\eeqa
Notice that all quantities $K_j$, $L_j$, $j=1,...,5$ are positive. 

The equations of motion in \equ{diss32} without the dissipation can be written as the first order system of equations 
\beqa{eqnew}
    \dot\varphi_1 &=& \frac{2}{C_1} J_1 \label{eq:aver-sys-conser-32-32}\nonumber\\
    \dot\varphi_2 &=& \frac{2}{C_2} J_2 \nonumber\\
    \dot J_1 &=& - K_1\sin(\varphi_1) - \chi \left\{K_2 \sin(\varphi_1) + K_3\sin(2\varphi_1) +  K_4\sin(\varphi_1 + \varphi_2) + K_5\sin(\varphi_1 - \varphi_2) \right\} \nonumber \\ 
    \dot J_2 &=& - L_1\sin(\varphi_2) - \chi \left\{L_2 \sin(\varphi_1) + L_3\sin(2\varphi_2) +  L_4\sin(\varphi_1 + \varphi_2) - L_5\sin(\varphi_1 - \varphi_2) \right\}\ .\nonumber\\
\eeqa
The equilibrium points are located at $J_1=J_2 = 0$, while the angles $(\varphi_1,\varphi_2)=(\varphi_1^0,\varphi_2^0)$ are the solutions of the system \equ{eq:crit-conservative} using the values of the constants $K_j$, $L_j$, $j=1,...,5$, given in \equ{kappaelle}. 

We analyze the character of the equilibria by linearizing the equations of motion \eqref{eq:aver-sys-conser-32-32} as in Section~\ref{sec:res11}. For example, considering the critical point at $(\varphi_1, \varphi_2, J_1, J_2) = (0, 0, 0, 0)$, the linear system is given by \equ{EMlin} and we have the following result. 

\begin{proposition}
For any value of the eccentricity, the eigenvalues of the linearized spin-orbit and spin-spin motions for the (3:2,3:2) resonance with, respectively, $\chi=0$ or $\chi=1$ in \equ{eqnew} are purely imaginary. 
\end{proposition}

\begin{proof}
For the spin-orbit problem, $\chi=0$, and the spin-spin problem, $\chi =1$, we introduce the same notation as in \equ{ab12}, which leads to the secular equation \equ{seceq} with the same $\alpha$, $\beta$ as in \equ{ab2}. 
Since $\alpha>0$ and $\alpha^2-4\beta>0$ (see \equ{511bis}), then the eigenvalues are purely imaginary, independently of the value of the eccentricity. 
\end{proof}

An example of the computation for the maximum of the real part of the eigenvalues associated to the (3:2,3:2) resonance, for different values of inclination and eccentricity, is given in Figure~\ref{fig:res1132} (left panel). 

\begin{figure}[h]
	\center
	\includegraphics[trim = 5cm  3cm 2cm  3cm, scale=0.15]{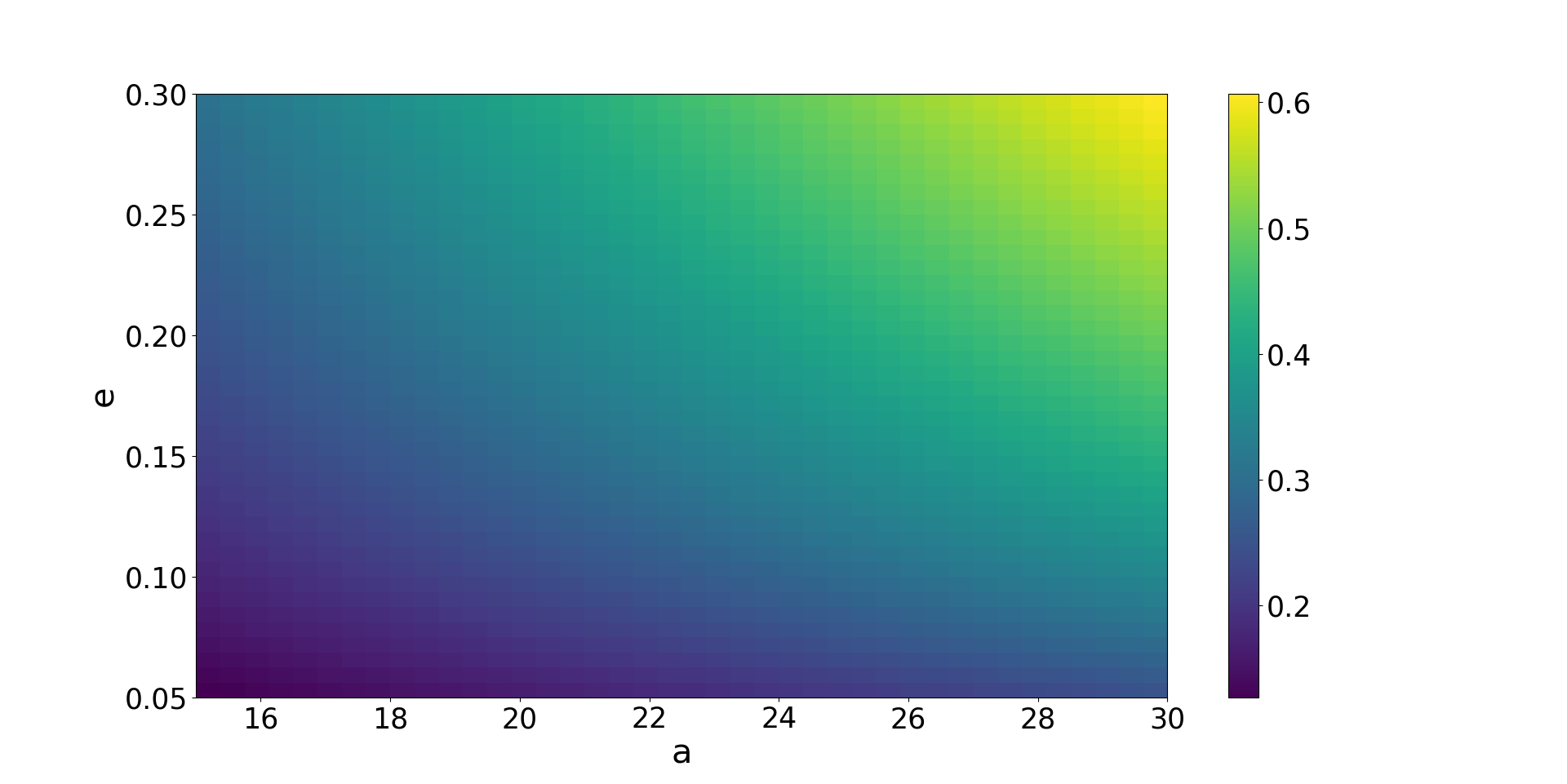}
	\includegraphics[trim = 5cm  3cm 5cm  3cm, scale=0.15]{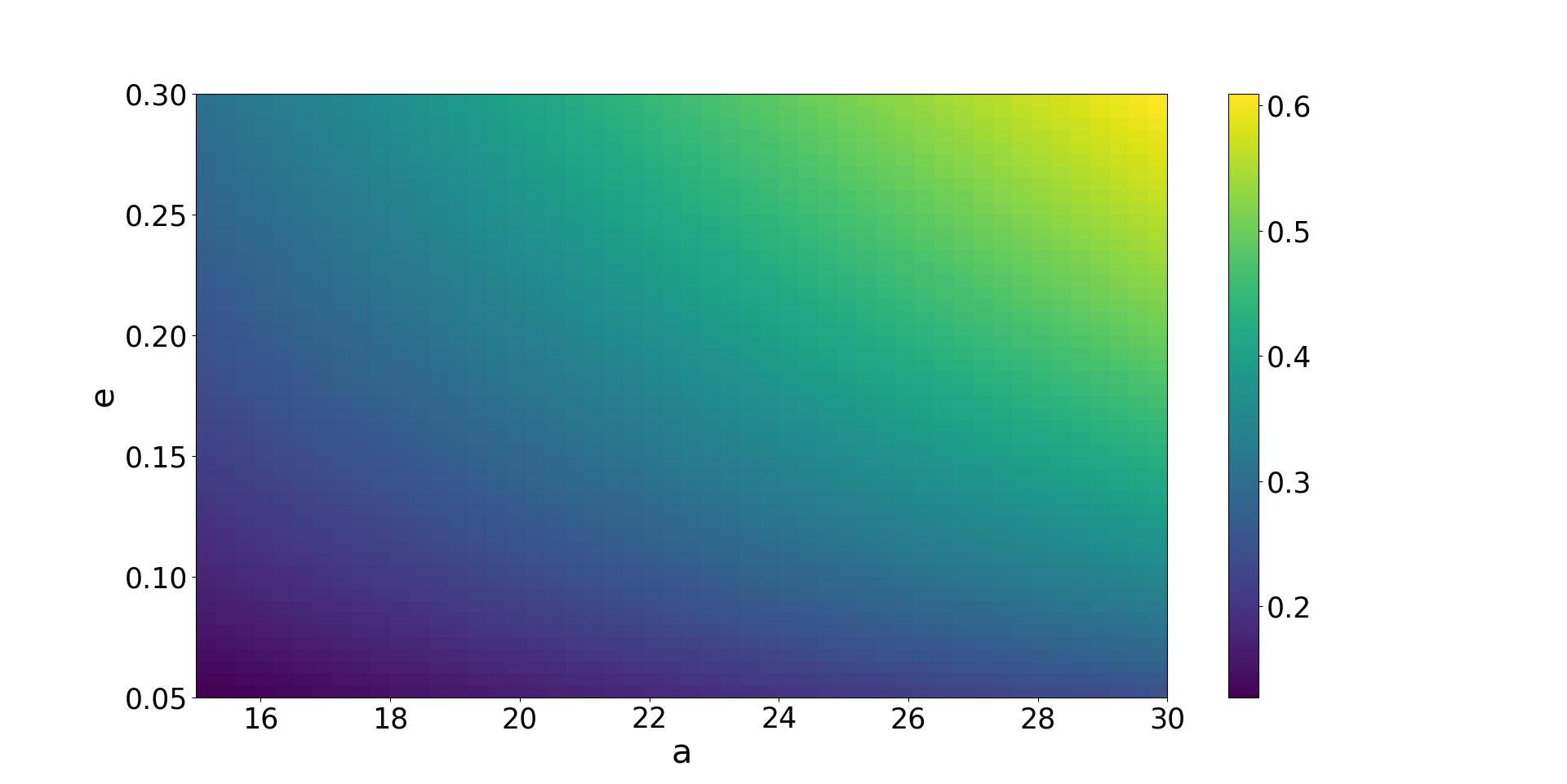}
	\caption{Keplerian spin-spin problem, maximum of the real part of the eigenvalues on a color scale for the conservative case, $(0,\pi)$: (3:2,3:2) resonance (left panel), (1:1,3:2) resonance (right panel). 
	}\label{fig:res1132}
\end{figure}

To deal with the (1:1,3:2) spin-spin resonances, we introduce the angles $$ \varphi_1 = 2\theta_1-2t, \qquad \varphi_2 = 2\theta_2 - 3t\ .$$
The equations of motion considering just the terms that depend only on $\varphi_1$ and $\varphi_2$, and their combinations, namely the terms that will not vanish after averaging with respect to time, become
\begin{align*}
	\frac{C_1}{2}\ddot\varphi_1 + K_1 \sin(\varphi_1) + \chi \{ K_2\sin(\varphi_1) + K_3\sin(2\varphi_1) + K_4\sin(\varphi_1 + \varphi_2) + K_5\sin(\varphi_1 - \varphi_2) \} \\= -\bar\gamma_1\left(\frac{\dot\varphi_1}{2} + 1 -\bar\mu_1\right) \\
	\frac{C_2}{2}\ddot\varphi_2 + L_1 \sin(\varphi_2) + \chi \{ L_2\sin(\varphi_2) + L_3\sin(2\varphi_2) + L_4\sin(\varphi_1 + \varphi_2)  - L_5\sin(\varphi_1 -\varphi_2) \} \\ = -\bar\gamma_2\left(\frac{\dot\varphi_2}{2} + \frac{3}{2} -\bar\mu_2\right)\ ,
\end{align*}
where, from Appendix \eqref{sub:potential},  we have  
\beqano
K_1 &=& \frac{3d_1GM_2}{2a^3}(1-{5\over 2}e^2),\qquad\qquad\ \ \ L_1 = \frac{21d_2eGM_1}{4a^3}\nonumber\\
K_2 &=&\frac{75d_1GM_2q_1}{56a^5M_1} + \frac{75d_1e^2GM_2q_1}{56a^5M_1} + \frac{15d_1Gq_2}{8a^5} + \frac{15d_1e^2Gq_2}{8a^5}\nonumber\\
L_2 &=& \frac{135d_2eGq_1}{16a^5} + \frac{675d_2eGM_1q_2}{112a^5M_2}, \nonumber\\
K_3 &=& \frac{75d_1^2GM_2}{16a^5M_1}(1-11e^2), \qquad\qquad\ \ L_3 = \frac{3825d_2^2e^2GM_1}{32a^5M_2}\nonumber\\
K_4 &=& \frac{1365d_1d_2eG}{32a^5}, \qquad\qquad\qquad\qquad L_4 = \frac{1365d_1d_2eG}{32a^5}\nonumber\\
K_5 &=&\frac{45d_1d_2eG}{32a^5},\qquad\qquad\qquad\qquad\quad L_5 = \frac{45d_1d_2eG}{32a^5}\ .
\eeqano 
Notice that $K_1>0$ for $e<\sqrt{2\over 5}$ and $K_3>0$ for $e<{1\over {\sqrt{11}}}$, while all other coefficients are positive.

The equilibrium points are located at $J_1=J_2 = 0$, while the angles are the solutions of the system \equ{eq:crit-conservative}. 

Considering  the critical point at $(\varphi_1, \varphi_2, J_1, J_2) = (0, 0, 0, 0)$, the linear system is 
\beqa{eq2}
    \dot\xi_1 &=& \frac{2}{C_1}J_1 \nonumber\\
    \dot\xi_2 &=& \frac{2}{C_2}J_2 \nonumber\\
    \dot J_1 &=& -\left\{ (K_1 + \chi K_2) +2\chi K_3 + \chi K_4 + \chi K_5 \right\}\xi_1 - \chi\left\{ K_4  - K_5 \right\} \xi_2 \nonumber\\
    \dot J_2 &=& - \chi\left\{ L_4 - L_5 \right\} \xi_1  -\left\{ (L_1 + \chi L_2) + 2\chi L_3 +\chi L_4 + \chi L_5  \right\}\xi_2\ . 
\eeqa
Then, we have the following results. 

\begin{proposition}
For $e<\sqrt{2\over 5}$, the eigenvalues of the linearized Keplerian spin-orbit motion for the $(1:1, 3:2)$ resonance with $\chi=0$ in \equ{eq2} are purely imaginary.
\end{proposition}

\begin{proposition}
For $e<{1\over {\sqrt{11}}}$, the eigenvalues of the linearized Keplerian spin-spin motion for the $(1:1,3:2)$ resonance with $\chi=1$ in \equ{eq2} are purely imaginary. 
\end{proposition}

The proof follows from the observation that as in \equ{ab12} the quantities $a_1$, $a_2$, $b_1$, $b_2$ are negative for $e<{1\over {\sqrt{11}}}$. 

An example of the computation for the maximum of the real part of the eigenvalues associated to the (1:1,3:2) resonance, for different values of inclination and eccentricity, is given in Figure~\ref{fig:res1132} (right panel). The results are very similar to those of the (3:2, 3:2) resonance.

\section{Comparison of different dynamical models}\label{sec:comparison}
The analysis of the equilibria and their linear stability of 
Sections \ref{sec:res11} and \ref{sec:res32}
are based on the Keplerian spin-spin problem, i.e. assuming that the orbit is Keplerian. Within such models, it is interesting to compare the behavior of the Keplerian spin-orbit problem
with $\chi=0$ and the Keplerian spin-spin problem in which also the
contribution of $V_4$ is considered; such comparison provides the coupling between the
rotational motions of the two ellipsoids (see Section \ref{sec:71}). 
Furthermore, releasing the assumption of Keplerian orbit, we will study the structure of the phase space in the full problem, i.e. when the coupling perturbs also the orbit (see Section \ref{sec:72}).

\subsection{A comparison between the spin-orbit and spin-spin problems}\label{sec:71}
We briefly provide some results on the comparison between 
the spin-orbit and the spin-spin problems. 
Figure \ref{fig:soss11} gives an example for the case of the binary asteroid Patroclus-Menoetius in the
conservative case (first and second row); the upper figures show the Poincar\'e maps in
components $(\theta_1,p_1)$ (left panel) and $(\theta_2,p_2)$ (right panel) for
the decoupled case $V=V_2$, while the panels in the second row show the results when the
coupling is switched on, namely $V=V_2+V_4$. In this sample case, the coupling between the rotational
motions provokes small differences in the plots, mainly visible within the
librational regions. For the Keplerian spin-spin problem in the mixed case with
dissipation only in the equations for $\E_1$, namely $\bar\gamma_1=10^{-3}$ and $\bar\gamma_2=0$ (Figure \ref{fig:soss11}, third row), we
notice that the plots for $(\theta_1,p_1)$ show an attractor toward the 1:1
resonance. Although the dissipation is not acting on $\E_2$, we see that the
second ellipsoid is also affected by the dissipation due to the coupling
obtained taking the potential $V=V_2+V_4$; this conclusion can be inferred by the small drift, typical of the dissipative plots. 
Finally, we provide an example for the dissipative case, with the two bodies both attracted to the $(1:1,1:1)$ 
resonance, as shown in the last row of Figure \ref{fig:soss11}.

\begin{figure}[h]
	\center
	\includegraphics[height=5.0cm]{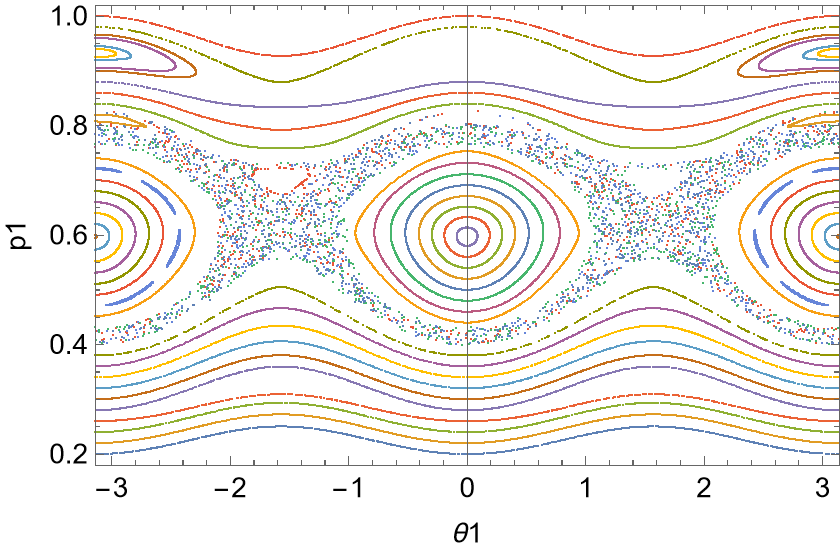}
	\includegraphics[height=5.0cm]{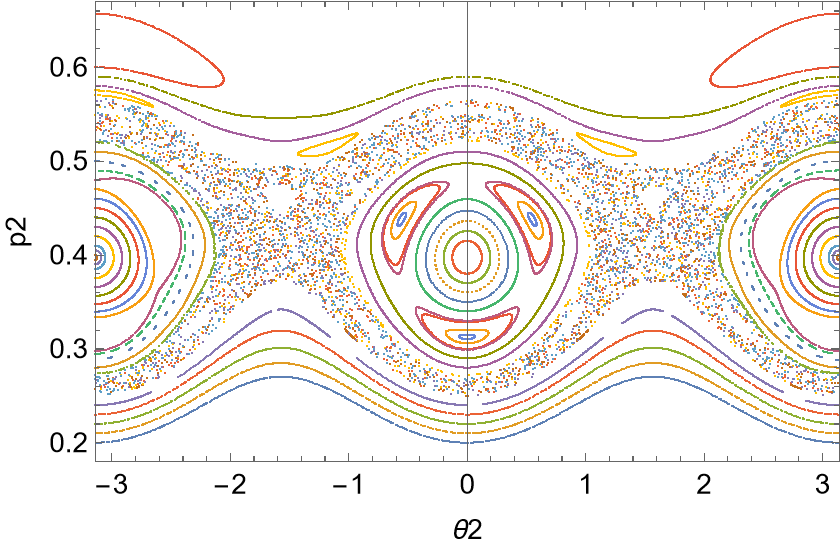}\\
	\includegraphics[height=5.0cm]{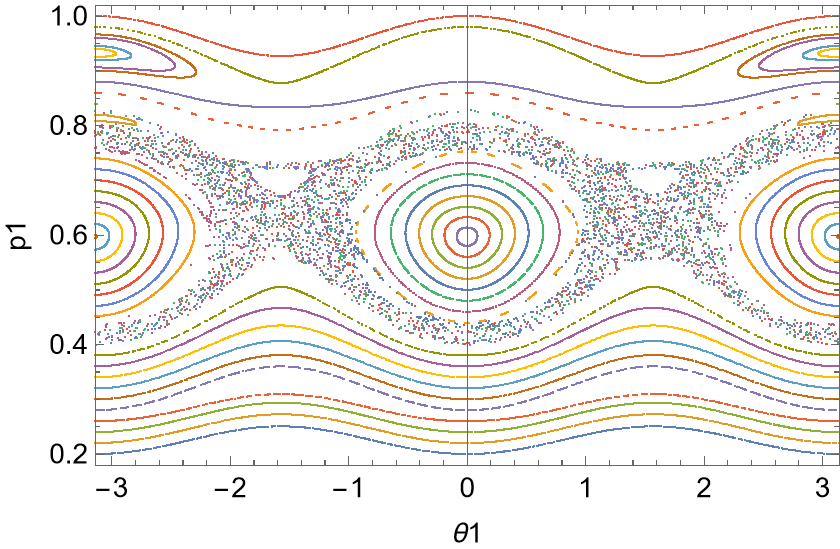}
	\includegraphics[height=5.0cm]{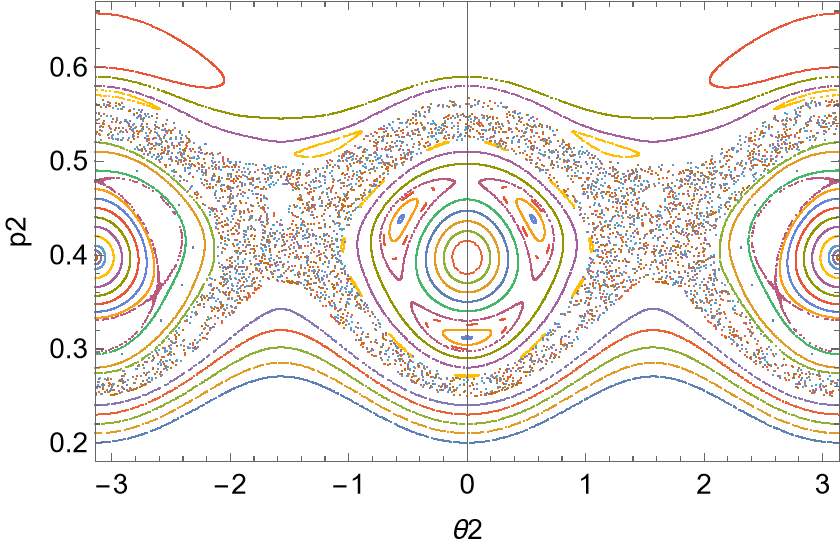}\\
	\includegraphics[height=5.0cm]{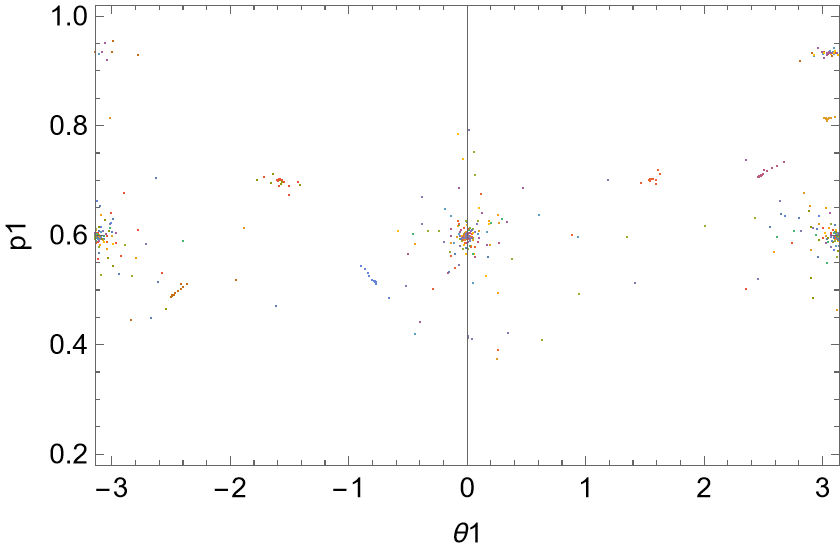}
	\includegraphics[height=5.0cm]{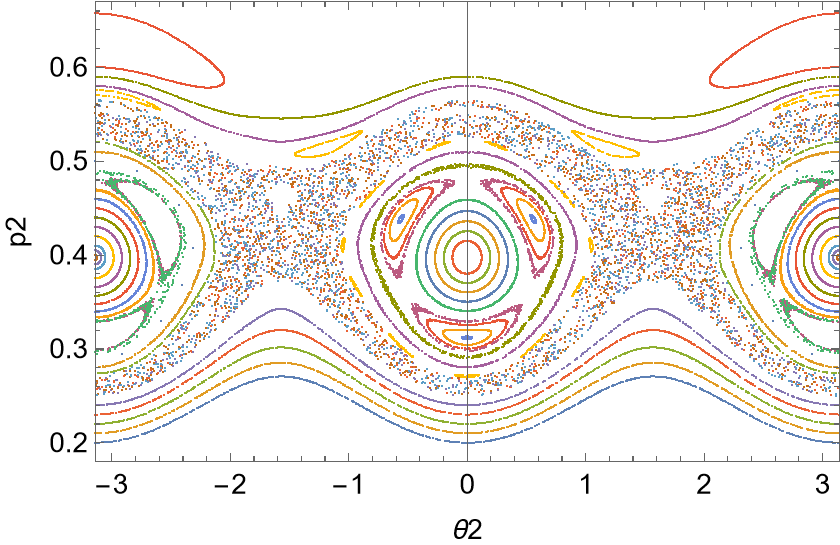}\\
	\includegraphics[height=5.0cm]{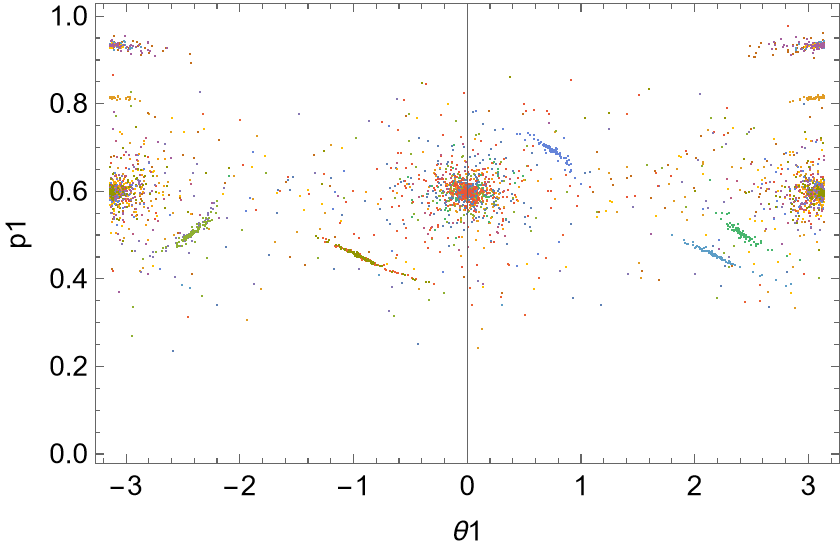}
	\includegraphics[height=5.0cm]{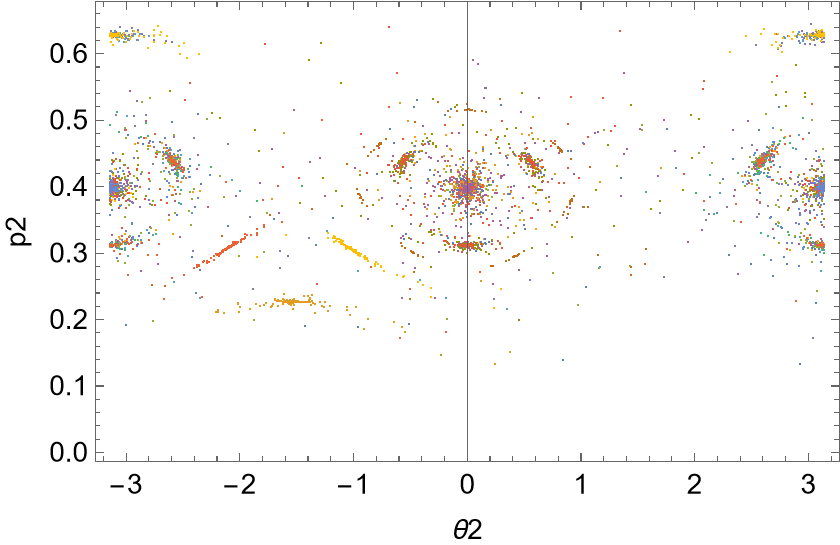}\\
	\caption{Poincar\'e sections projected on $(\theta_1,p_1)$, $(\theta_2,p_2)$ for Patroclus-Menoetius for the Keplerian
	spin-orbit problem (upper panels) and the Keplerian spin-spin problem
	(second row). Mixed case with $\bar\gamma_1=10^{-3}$ (third row).
	Dissipative case with $\bar\gamma_1=10^{-4}$ and $\bar\gamma_1=10^{-4}$
	(last row). }\label{fig:soss11}
\end{figure}

\subsection{Numerical investigation of the full spin-spin problem}\label{sec:72}
In this Section, we proceed to a numerical study of the full spin-spin problem.
For this reason we compare the solutions of the Keplerian case  determined by \equ{orbital_eqs_kep} with the corresponding solutions of the full problem given
by \equ{orbital_eqs} and \equ{spin-spin_V}. To compare the solutions in presence 
of dissipation, we add the right hand sides of \equ{diss} to the corresponding equations in \equ{spin-spin_V}, so that the dissipation 
essentially affects the momenta $p_j$. 

In the following, several issues need to
be addressed. First, we notice that in the full problem also the orbital parameters, 
i.e. the semi-major axis $a$ and the eccentricity $e$, are subject to change with time.  
In presence of dissipation (see \equ{diss}) it is therefore also necessary to 
calculate $a=a(t)$ at each integration step. Since the orbital part of the
dynamics is given in polar coordinates, we require a simple relation
between the coordinates $(r,f,p_r,p_f)$ and the time dependent parameter $a$.
To this end, starting with the definition of the orbital energy 
\beqno
-\frac{G M}{2a} = \frac{\dot r^2 + r^2\dot f^2}{2} - \frac{GM}{r} \ ,
\eeqno
the parameter $a$ is given in terms of $r$, $p_r$, and $p_f$ by 
\beq{orb_a}
a = -\frac{a_0^3 M_1^2 M_2^2 r^2}{p_f^2-2a_0^3M_1^2M_2^2r+p_r^2r^2} \ ,
\eeq
where we have used $p_r=\mu \dot r$, $p_f=\mu r^2\dot f$, together with $\mu=M_1
M_2$.  We notice, that the initial value of the semi-major axis $a_0=a(0)$
enters the relation due to our choice of units, i.e. from the relation $G M =
a_0^3$ that needs to be kept fixed. 
Second, we require a suitable section condition to compare the results of the (fixed orbit) Keplerian problem with the full problem, where the orbit is subject to a 
change with time. We notice that in the Keplerian case the section condition is 
taken to be $t\mod 2\pi=0$, which means to take values of the solution state whenever 
the bodies complete one full$\bar\gamma_1=10^{-3}$ revolution period. However, in the full problem, 
since $a=a(t)$, the orbital period is not constant anymore and an alternative section 
condition is needed. Setting the initial true anomaly $f(0)=0$, at time zero, the 
bodies start at their pericenters, and return to their pericenters after a time $2\pi$ 
in the Keplerian case.  This directly translates into the condition $r\sin(f)=0$ 
(crossing from below) in the full problem. Thus, surfaces of section in the full
problem are taken whenever the bodies are located at their pericenters, as it is
done in the case of fixed Keplerian orbits. 

As first test, we take the same initial 
conditions and parametes of the Keplerian case in the conservative setting, but now based on the dynamics of the full problem. 
With a higher number of degrees of freedom, the integration becomes more computationally complex; to this end, we use a Runge-Kutta, variable 
stepsize integration method. To check the accuracy, we control the preservation of
the Hamiltonian of the full problem, which turns out to be conserved up to $10^{-10}$ over the full integration time.  The results of this first simulation are shown in the
upper row of Figure \ref{fig:soss11B}, which should be compared with the second
row of Figure \ref{fig:soss11}.  While higher order resonant islands appear in
the Keplerian case, most of them are destroyed in the full problem. Moreover, additional
oscillations in the projections of the solutions to the phase portraits are visible
in Figure \ref{fig:soss11B} first row, that are not present in the Keplerian case.  We
conjecture that they are stemming from the variations of the orbital parameters
of the full problem. Despite these main differences, we argue that the dynamics 
near the main libration centers aligns remarkably well with the expected location 
and geometry when comparing Figures \ref{fig:soss11} and \ref{fig:soss11B}.

In the second row of Figure \ref{fig:soss11B}, we also report the evolution of the 
orbital elements, i.e. semi-major axis $a$ and eccentricity $e$, over 1000 orbital 
revolution periods. The elements $a$, $e$ follow regular and oscillatory motions and 
stay close to their initial conditions, as shown in Figure \ref{fig:soss11B}, second row. 

Next, we investigate the long-term behaviour of the solution in the dissipative case. We take two different initial conditions in $(p_1,\theta_1)$, $(p_2,\theta_2)$ and perform a
long term integration over 100\,000 orbital revolution periods of the two bodies; the black dots mark the final end states.  
The results are given in the 3rd row of Figure \ref{fig:soss11B} and show that all initial conditions tend towards attractors, marked by black points in both panels, very close to the location of the unperturbed (1:1,1:1) resonance. 

We complement these results with those of the fourth row of Figure \ref{fig:soss11B}, providing semimajor axis and eccentricity as a function of time. We observe that, during the orbital evolution the trajectory experiences a transient irregular phase, followed by  
a regular oscillatory behaviour when the spin motions are finally approaching the final attractor. In terms of the orbital elements, the end state is an oscillation around the initial values of $a$ and $e$,  while depending on the choice of the initial
condition of the spin dynamics. 


We also repeated the same set of simulations for the parameters of the Pluto-Charon
system.  The comparison between the full and Keplerian case lead to the same conclusions 
as for the Patroclus-Menoetius case.  However, it turns out that the differences between 
the Keplerian and full problems are much less evident: the phase portraits (not shown 
here) are in better agreement.  Moreover, the dynamics of the orbital elements stays closer to the initial values. We conclude that the magnitudes of the conservative and dissipative parameters are relevant to see any difference between the Keplerian and full problems; hence, as it is natural to expect, the Keplerian case is a good approximation of the full problem, provided the parameters are small enough.

\begin{figure}[h]
	\center
\includegraphics[height=4.08cm]{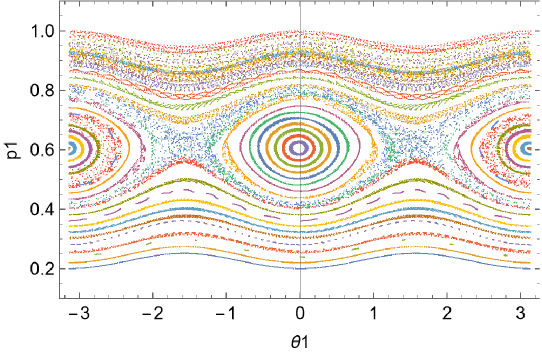}
\includegraphics[height=4.08cm]{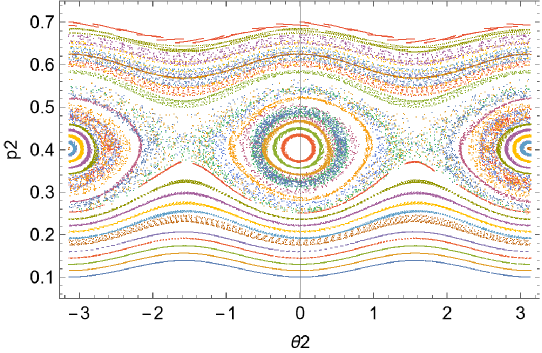}\\
	\includegraphics[height=4.0cm]{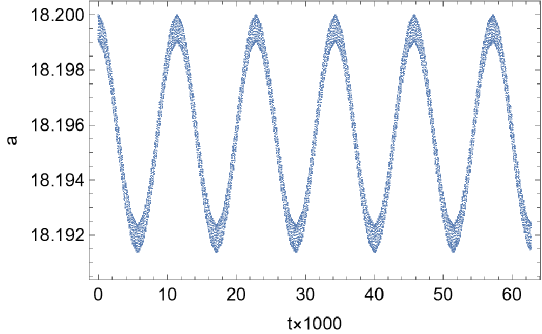}
	\includegraphics[height=4.0cm]{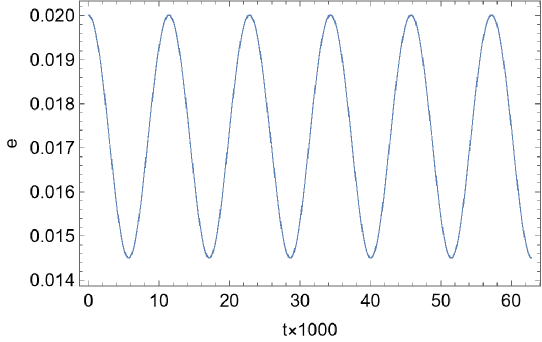}\\
	\includegraphics[height=4.08cm]{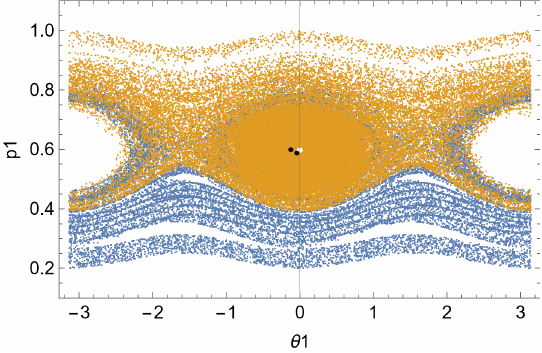}
	\includegraphics[height=4.08cm]{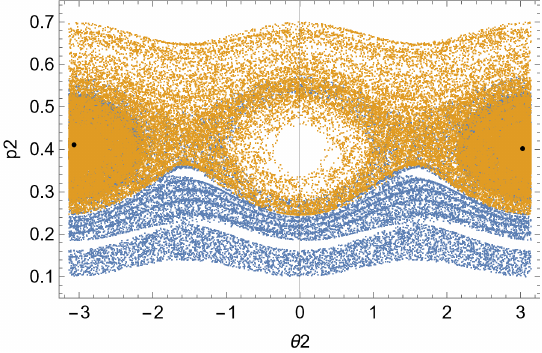}\\
	\includegraphics[height=4.08cm]{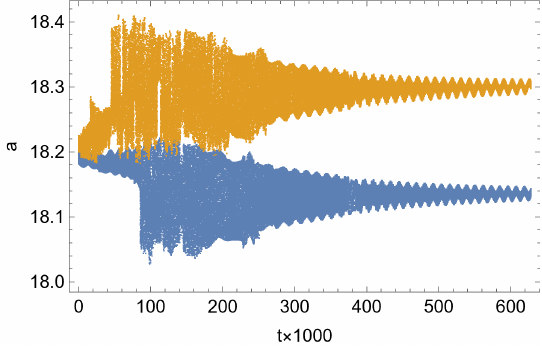}
	\includegraphics[height=4.08cm]{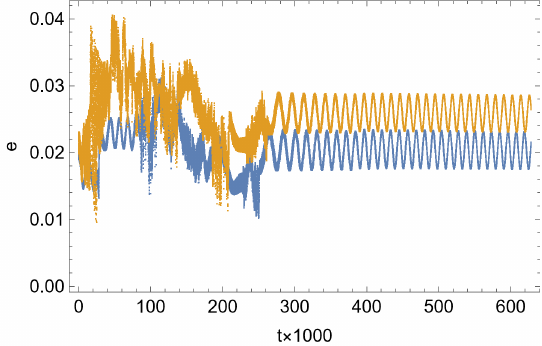}\\
	\caption{Comparison between conservative (1st and 2nd row) and dissipative (3rd and 4th row) dynamics in the full spin-spin problem for Patroclus-Menoetius.\\
	1st row: Poincar\'e sections projected on $(\theta_1,p_1)$, $(\theta_2,p_2)$ in the conservative case for the same initial conditions and parameters as shown in the second row in Fig. 3.\\
	2nd row: evolution of semi-major axis $a$ and eccentricity $e$ for the initial condition $(p_1,\theta_1)=(0.6,0)$, $(p_2,\theta_2)=(0.4,0)$ (close to the centers of the librational islands).\\
	3rd row: long-term integration for two specific initial conditions,
	$(p_1,\theta_1)=(0.2,0)$, $(p_2,\theta_2)=(0.4,0)$ (blue) and $(p_1,\theta_1)=(1,0)$, $(p_2,\theta_2)=(0.4,0)$ (yellow) in the
	dissipative case with $\bar\gamma_1=6\times10^{-6}$ and $\bar\gamma_2=4\times10^{-6}$; the black dots indicate the state of the orbit at the end of the integration time.\\
	4th row: shows the corresponding evolution of $a$ and $e$ for the
	orbits shown in row 3.}\label{fig:soss11B}
\end{figure}

\section{Conclusions}\label{sec:conclusions}
Given the increased interest in the exploration of minor bodies of the Solar system, the investigation of rotational dynamics plays a relevant role. In this work, we analyzed the rotational dynamics of a binary system under different models of increasing complexity: $(i)$ the spin-orbit problem, in which the rotational dynamics of the two bodies are decoupled and the orbits are assumed to be Keplerian ellipses; $(ii)$ the Keplerian spin-spin problem, in which there is an interaction between the rotational dynamics of the two bodies and the orbits are still assumed to be fixed ellipses; $(iii)$ the full spin-spin problem, in which there is an interaction of the rotational dynamics and also a coupling with the orbit, which is not constrained to be a Keplerian ellipse. For the main resonances, namely the $(1:1,1:1)$, $(3:2,3:2)$, $(1:1,3:2)$ resonances, we produced analytical results by studying the linear stability of the different models, and numerical computations by integrating the equations of motion. 

The investigation based on the linear stability for a sample equilibrium point, namely the origin, showed that the results are different for the spin-orbit and the Keplerian spin-spin problem, since in the former case one has imaginary eigenvalues for any eccentricity less than $\sqrt{2/5}$, while in the latter case one has linear stability for eccentricity less than $1/\sqrt{11}$; in the dissipative case, one needs to give constraints on the dissipative factors in order to get the stability of the origin. Same results are obtained for the $(1:1,3:2)$ resonance, while a different behavior is observed for the $(3:2,3:2)$ resonance for which the origin is linearly stable for any value of the eccentricity. 

We also added some results on the numerical integration of the Keplerian spin-orbit and spin-spin problems, noticing an interaction between the rotational states of the two bodies, which is observed especially in the librational regions. When adding the dissipation on both bodies, they tend towards the respective attractors; when the dissipation acts only on one body, the dissipative effect acts on both bodies, although at a minor extent for the body on which the dissipation is not acting directly. 

Another interesting result concerns the numerical integration of the full spin-spin problem in which also the orbit is perturbed and the orbital elements, in particular the semimajor axis and the eccentricity, vary with time. In particular, the effect of the variation with time of the trajectory provokes a disappearance of higher order islands. Finally, in our sample case of the binary asteroid Patroclus and Menoetius, the dissipation leads to attractors close to the unperturbed location of the $(1:1,1:1)$ resonance together with oscillations of semimajor axis and eccentricity.

\appendix

\section{Computation of $\bar\gamma_1$, $\bar\gamma_2$, $\bar\mu_1$, $\bar\mu_2$ in \equ{gm}}\label{app:gm}

We remind the following formulae from Kepler's motion (in our units):
\beqano
n&=&{{2\pi}\over T}=\sqrt{G}\ a^{-{3\over 2}}=1\nonumber\\
r&=&{{a(1-e^2)}\over {1+e \cos f}}\nonumber\\
\dot f&=&{{df}\over {dt}}={{df}\over {d\ell}}{{d\ell}\over {dt}}=n{{df}\over {d\ell}}\nonumber\\
{{d\ell}\over {df}}&=&{n\over {\dot f}}={{r^2}\over
{a^2\sqrt{1-e^2}}}\ . \eeqano Then, we have the following relation
for the average of $(a/r)^6$: 
\beqano
\langle ({a\over {r(t)}})^6\rangle&=&{1\over {2\pi}} \int_0^{2\pi}({a\over {r(t)}})^6d\ell\nonumber\\
&=&{1\over {2\pi}} \int_0^{2\pi}({a\over {r(t)}})^6{{r^2}\over {a^2\sqrt{1-e^2}}}df\nonumber\\
&=&{1\over {2\pi}} {1\over {(1-e^2)^{9\over 2}}} \int_0^{2\pi} (1+e\cos f)^4 df\nonumber\\
&=&{1\over {(1-e^2)^{9\over 2}}}\ (1+3e^2+{3\over 8}e^4)\ .
\eeqano
In a similar way, we obtain the following relation for the average of $(a/r)^6\,\dot f$:
\beqano
\langle ({a\over {r(t)}})^6\,\dot f\rangle&=&{1\over {2\pi}} \int_0^{2\pi}({a\over {r(t)}})^6\,\dot fd\ell\nonumber\\
&=&{1\over {2\pi}} \int_0^{2\pi}({a\over {r(t)}})^6{{d\ell}\over {df}}\, n\, {{df}\over {d\ell}}\,df\nonumber\\
&=&{1\over {2\pi}} n\, \int_0^{2\pi}{{(1+e\cos f)^6}\over {(1-e^2)^6}} df\nonumber\\
&=&{n\over {(1-e^2)^6}}\ (1+{{15}\over 2}e^2+{{45}\over 8}e^4+{5\over {16}}e^6)\ .
\eeqano
From \equ{diss}, we obtain the averaged equations with
$$
\bar\gamma_j=\delta_j{ C_j}\ \langle ({a\over {r(t)}})^6\rangle=\delta_j{ C_j}
{1\over {(1-e^2)^{9\over 2}}}\ (1+3e^2+{3\over 8}e^4)\ ,\qquad j=1,2
$$
and
\beqano
\bar\mu_j&=&{{\delta_j{ C_j}\ \langle ({a\over {r(t)}})^6\,\dot f\rangle}\over {\delta_j{ C_j}\ \langle ({a\over {r(t)}})^6\rangle}}\nonumber\\
&=&{n\over {(1-e^2)^{3\over 2}}}\ {{1+{{15}\over 2}e^2+{{45}\over 8}e^4+{5\over {16}}e^6}\over {1+3e^2+{3\over 8}e^4}}\ ,\qquad j=1,2\ .
\eeqano
Notice that $\bar\gamma_1\not=\bar\gamma_2$ whenever $\delta_1 C_1\not=\delta_2 C_2$, while $\bar\mu_1=\bar\mu_2$.

\section{Expansion of the potential to the second order in the eccentricity}\label{sub:potential}
We give below the expansion of $V_2$ and $V_4$ in \equ{V024} up to second order in
the eccentricity:

\begin{eqnarray*}
V_2 &=& -\frac{G M_2 q_1}{4 a^3} - \frac{3 e^2 G M_2 q_1}{8 a^3} - \frac{G M_1 q_2}{4 a^3} - \frac{3 e^2 G M_1 q_2}{8 a^3} - \frac{3 e G M_2 q_1 \cos(t)}{4 a^3} - \frac{3 e G M_1 q_2 \cos(t)}{4 a^3}\\
&-& \frac{9 e^2 G M_2 q_1 \cos(2 t)}{8 a^3} - \frac{9 e^2 G M_1 q_2 \cos(2 t)}{8 a^3} + \frac{3 d_1 e G M_2 \cos(t - 2 \theta_1)}{8 a^3}\\
&-& \frac{3 d_1 G M_2 \cos(2 t - 2 \theta_1)}{4 a^3} + \frac{15 d_1 e^2 G M_2 \cos(2 t - 2 \theta_1)}{8 a^3} - \frac{21 d_1 e G M_2 \cos(3 t - 2 \theta_1)}{8 a^3}\\
&-& \frac{51 d_1 e^2 G M_2 \cos(4 t - 2 \theta_1)}{8 a^3} + \frac{3 d_2 e G M_1 \cos(t - 2 \theta_2)}{8 a^3} - \frac{3 d_2 G M_1 \cos(2 t - 2 \theta_2)}{4 a^3}\\
&+& \frac{15 d_2 e^2 G M_1 \cos(2 t - 2 \theta_2)}{8 a^3} - \frac{21 d_2 e G M_1 \cos(3 t - 2 \theta_2)}{8 a^3} - \frac{51 d_2 e^2 G M_1 \cos(4 t - 2 \theta_2)}{8 a^3}
\end{eqnarray*}

\newpage
\begin{eqnarray*}
V_4 & = & -\frac{45 d_2^2 G M_1}{448 a^5 M_2} - \frac{225 d_2^2 e^2 G M_1}{448 a^5 M_2} - \frac{45 d_1^2 G M_2}{448 a^5 M_1} - \frac{225 d_1^2 e^2 G M_2}{448 a^5 M_1} - \frac{45 G M_2 q_1^2}{224 a^5 M_1}\\
&-& \frac{225 e^2 G M_2 q_1^2}{
 224 a^5 M_1} - \frac{9 G q_1 q_2}{16 a^5} - \frac{45 e^2 G q_1 q_2}{16 a^5} - \frac{
 45 G M_1 q_2^2}{224 a^5 M_2} - \frac{225 e^2 G M_1 q_2^2}{224 a^5 M_2} \\
 &-& \frac{225 d_2^2 e G M_1 \cos(t)}{448 a^5 M_2} - \frac{225 d_1^2 e G M_2 \cos(t)}{448 a^5 M_1} - \frac{225 e G M_2 q_1^2 \cos(t)}{224 a^5 M_1} - \frac{ 45 e G q_1 q_2 \cos(t)}{16 a^5}\\
 &-& \frac{225 e G M_1 q_2^2 \cos(t)}{224 a^5 M_2} - \frac{225 d_2^2 e^2 G M_1 \cos(2 t)}{224 a^5 M_2} - \frac{225 d_1^2 e^2 G M_2 \cos(2 t)}{224 a^5 M_1}\\ 
 &-& \frac{ 225 e^2 G M_2 q_1^2 \cos(2 t)}{112 a^5 M_1} - \frac{45 e^2 G q_1 q_2 \cos(2 t)}{8 a^5} - \frac{225 e^2 G M_1 q_2^2 \cos(2 t)}{112 a^5 M_2} \\ 
 &-& \frac{75 d_1^2 e^2 G M_2 \cos(2 t - 4 \theta_1)}{128 a^5 M_1} + \frac{225 d_1^2 e G M_2 \cos(3 t - 4 \theta_1)}{128 a^5 M_1} - \frac{75 d_1^2 G M_2 \cos(4 t - 4 \theta_1)}{64 a^5 M_1}\\
 &+& \frac{825 d_1^2 e^2 G M_2 \cos(4 t - 4 \theta_1)}{64 a^5 M_1} - \frac{975 d_1^2 e G M_2 \cos(5 t - 4 \theta_1)}{128 a^5 M_1} - \frac{3825 d_1^2 e^2 G M_2 \cos(6 t - 4 \theta_1)}{128 a^5 M_1}\\
 &-& \frac{75 d_1 e G M_2 q_1 \cos(t - 2 \theta_1)}{224 a^5 M_1} - \frac{15 d_1 e G q_2 \cos(t - 2 \theta_1)}{32 a^5} - \frac{75 d_1 G M_2 q_1 \cos(2 t - 2 \theta_1)}{112 a^5 M_1}\\
 &-& \frac{75 d_1 e^2 G M_2 q_1 \cos(2 t - 2 \theta_1)}{112 a^5 M_1} - \frac{15 d_1 G q_2 \cos(2 t - 2 \theta_1)}{16 a^5} - \frac{15 d_1 e^2 G q_2 \cos(2 t - 2 \theta_1)}{16 a^5}\\
 &-& \frac{675 d_1 e G M_2 q_1 \cos(3 t - 2 \theta_1)}{224 a^5 M_1} - \frac{135 d_1 e G q_2 \cos(3 t - 2 \theta_1)}{32 a^5} - \frac{3975 d_1 e^2 G M_2 q_1 \cos(4 t - 2 \theta_1)}{448 a^5 M_1}\\
 &-& \frac{795 d_1 e^2 G q_2 \cos(4 t - 2 \theta_1)}{64 a^5} - \frac{225 d_1 e^2 G M_2 q_1 \cos(2 \theta_1)}{448 a^5 M_1} - \frac{45 d_1 e^2 G q_2 \cos(2 \theta_1)}{64 a^5}\\
 &-& \frac{75 d_2^2 e^2 G M_1 \cos(2 t - 4 \theta_2)}{128 a^5 M_2} + \frac{225 d_2^2 e G M_1 \cos(3 t - 4 \theta_2)}{128 a^5 M_2} - \frac{75 d_2^2 G M_1 \cos(4 t - 4 \theta_2)}{64 a^5 M_2}\\
 &+& \frac{825 d_2^2 e^2 G M_1 \cos(4 t - 4 \theta_2)}{64 a^5 M_2} - \frac{975 d_2^2 e G M_1 \cos(5 t - 4 \theta_2)}{128 a^5 M_2} - \frac{3825 d_2^2 e^2 G M_1 \cos(6 t - 4 \theta_2)}{128 a^5 M_2}\\ 
 &-& \frac{15 d_2 e G q_1 \cos(t - 2 \theta_2)}{32 a^5} - \frac{75 d_2 e G M_1 q_2 \cos(t - 2 \theta_2)}{224 a^5 M_2} - \frac{15 d_2 G q_1 \cos(2 t - 2 \theta_2)}{16 a^5}\\
 &-& \frac{15 d_2 e^2 G q_1 \cos(2 t - 2 \theta_2)}{16 a^5} - \frac{75 d_2 G M_1 q_2 \cos(2 t - 2 \theta_2)}{112 a^5 M_2} - \frac{75 d_2 e^2 G M_1 q_2 \cos(2 t - 2 \theta_2)}{112 a^5 M_2}\\
 &-& \frac{135 d_2 e G q_1 \cos(3 t - 2 \theta_2)}{32 a^5} - \frac{675 d_2 e G M_1 q_2 \cos(3 t - 2 \theta_2)}{224 a^5 M_2} - \frac{795 d_2 e^2 G q_1 \cos(4 t - 2 \theta_2)}{64 a^5}\\
 &-& \frac{3975 d_2 e^2 G M_1 q_2 \cos(4 t - 2 \theta_2)}{448 a^5 M_2} - \frac{105 d_1 d_2 e^2 G \cos(2 t - 2 \theta_1 - 2 \theta_2)}{64 a^5}\\
 &+& \frac{315 d_1 d_2 e G \cos(3 t - 2 \theta_1 - 2 \theta_2)}{64 a^5} - \frac{105 d_1 d_2 G \cos(4 t - 2 \theta_1 - 2 \theta_2)}{32 a^5}\\
 &+& \frac{1155 d_1 d_2 e^2 G \cos(4 t - 2 \theta_1 - 2 \theta_2)}{32 a^5} - \frac{1365 d_1 d_2 e G \cos(5 t - 2 \theta_1 - 2 \theta_2)}{64 a^5}\\
 &-& \frac{5355 d_1 d_2 e^2 G \cos(6 t - 2 \theta_1 - 2 \theta_2)}{64 a^5} - \frac{9 d_1 d_2 G \cos(2 \theta_1 - 2 \theta_2)}{32 a^5} - \frac{45 d_1 d_2 e^2 G \cos(2 \theta_1 - 2 \theta_2)}{32 a^5}\\
 &-& \frac{45 d_1 d_2 e G \cos(t + 2 \theta_1 - 2 \theta_2)}{64 a^5} - \frac{45 d_1 d_2 e^2 G \cos(2 t + 2 \theta_1 - 2 \theta_2)}{32 a^5} - \frac{45 d_2 e^2 G q_1 \cos(2 \theta_2)}{64 a^5}\\
 &-& \frac{225 d_2 e^2 G M_1 q_2 \cos(2 \theta_2)}{448 a^5 M_2} - \frac{45 d_1 d_2 e G \cos(t - 2 \theta_1 + 2 \theta_2)}{64 a^5} - \frac{45 d_1 d_2 e^2 G \cos(2 t - 2 \theta_1 + 2 \theta_2)}{32 a^5}
\end{eqnarray*}

\bibliographystyle{abbrv}
\bibliography{references}

\end{document}